\documentclass[a4article, DIV=12, 11pt]{scrartcl}
\usepackage[latin1]{inputenc}
\usepackage{amsmath}
\usepackage{amsfonts}
\usepackage{amssymb}
\usepackage{amsthm}
\usepackage{graphicx}
\usepackage{thmtools}
\usepackage{bbm}
\usepackage{hyperref}
\usepackage[bottom, marginal]{footmisc}
\usepackage{todonotes}

\usepackage{tikz}
\usetikzlibrary{matrix}

\declaretheoremstyle[headfont=\normalfont]{normalhead}
\newtheorem{lemma}{Lemma}[section]
\newtheorem{theorem}[lemma]{Theorem}
\newtheorem{proposition}[lemma]{Proposition}
\newtheorem{corollary}[lemma]{Corollary}
\newtheorem{definition}[lemma]{Definition}
\newtheorem{remark}[lemma]{Remark}

\newcounter{mt}

\newcommand{\R}{\mathbb{R}}
\newcommand{\C}{\mathbb{C}}

\DeclareMathOperator{\Val}{Val}
\DeclareMathOperator{\VConv}{VConv}
\DeclareMathOperator{\Conv}{Conv}
\DeclareMathOperator{\vol}{vol}

\DeclareMathOperator{\supp}{supp}

\DeclareMathOperator{\res}{res}

\DeclareMathOperator{\GL}{GL}

\DeclareMathOperator{\GW}{\mathrm{GW}}

\DeclareMathOperator{\Gr}{\mathrm{Gr}}

\DeclareMathOperator{\KL}{\mathrm{Kl}}

\DeclareMathOperator{\Sym}{\mathrm{Sym}}

\DeclareMathOperator{\SO}{\mathrm{SO}}

\DeclareMathOperator{\MAVal}{\mathrm{MAVal}}

\DeclareMathOperator{\MA}{\mathrm{MA}}

\renewcommand{\P}{\mathcal{P}}

\setkomafont{title}{\normalfont\Large}

\author{Jonas Knoerr}
\title{Monge-Amp\`ere operators and valuations}
\date{}

\newcommand{\Addresses}{{
		\bigskip
		\footnotesize
		
		Jonas Knoerr, \textsc{Institute of Discrete Mathematics and Geometry, TU Wien, Wiedner Hauptstrasse 8-10, 1040 Wien, Austria}\par\nopagebreak
		\textit{E-mail address}: \texttt{jonas.knoerr@tuwien.ac.at}
		
		\medskip
	}}
	
\makeatletter
\def\blfootnote{\xdef\@thefnmark{}\@footnotetext}
\makeatother

\makeindex
\begin{document}
\maketitle
\begin{abstract}
	Two classes of measure-valued valuations on convex functions related to Monge-Amp\`ere operators are investigated and classified. It is shown that the space of all valuations with values in the space of complex Radon measures on $\R^n$ that are locally determined, continuous, dually epi-translation invariant as well as translation equivariant, is finite dimensional. Integral representations of these valuations and a description in terms of mixed Monge-Amp\`ere operators are established, as well as a characterization of $\SO(n)$-equivariant valuations in terms of Hessian measures.
\end{abstract}

\section{Introduction}
	Monge-Amp\`ere-type equations arise in many problems from analysis and geometry, including Minkowski's problem, optimal transport, affine geometry, pluripotential theory, as well as physics, and play a prominent role in the Calabi-Yau Theorem~\cite{YauRiccicurvaturecompact1978}. In its classical form, the Monge-Amp\`ere equation is given by 
	\begin{align}
		\label{eq:MAEquation}
		\det(D^2f(x))=F(x,f(x),\nabla f(x)) \quad\text{for} ~x\in U,
	\end{align}
	where $U\subset\R^n$ is an open convex subset of $\R^n$, $D^2f$ denotes the Hessian of a convex function $f:U\rightarrow\R$, and $F:U\times\R\times\R^n\rightarrow\R^+$ is given.  We refer to the articles by Figalli \cite{FigalliMongeAmpereequation2019} and Trudinger and Wang \cite{TrudingerWangMongeAmpereequation2008} for a background on the Monge-Amp\`ere equation and its geometric applications.\\
	
	The notion of a generalized solution of the Monge-Amp\`ere equation goes back to Alexandrov. Here, the left hand side of \eqref{eq:MAEquation} is interpreted as a Borel measure $\MA(f)$ on $U$, called the \emph{Monge-Amp\`ere measure} of $f$, which extends by continuity to all finite-valued convex functions on $U$. More precisely, let $\Conv(U,\R)$ denote the space of all finite-valued convex functions on a convex open set $U\subset\R^n$ and $\mathcal{M}(U):= (C_c(U))'$ the space of complex Radon measures on $U$, considered as the continuous dual of the space $C_c(U)$ of complex-valued continuous functions with compact support. If we equip $\Conv(U,\R)$ with the topology induced by locally uniform convergence and $\mathcal{M}(U)$ with the weak*-topology, then the real Monge-Amp\`ere operator
	\begin{align*}
		\MA:\Conv(U,\R)&\rightarrow\mathcal{M}(U)\\
		f&\mapsto \MA(f)	
	\end{align*}
	is a continuous map and satisfies $d\MA(f)[x]=\det(D^2f(x))d\vol_n(x)$ for $f\in\Conv(U,\R)\cap C^2(U)$. Similar results hold for other operators of Monge-Amp\`ere-type, for example the complex Monge-Amp\`ere operator. \\
	One of the fundamental properties of these operators is that they define \emph{valuations} on $\Conv(U,\R)$. Here, a functional $\Psi$ defined on a family $X$ of (extended) real-valued functions with values in an Abelian semi-group is called a valuation if
	\begin{align*}
		\Psi(f)+\Psi(h)=\Psi(f\vee h)+\Psi (f\wedge h)
	\end{align*}
	for all $f,h\in X$ such that the pointwise maximum $f\vee h$ and minimum $f\wedge h$ belong to $X$. 
	This can be considered as a generalization of the notion of valuations on a family of sets, where the pointwise maximum and minimum are replaced by union and intersection. 
	In recent years, valuations on different families of sets, and especially convex bodies, have been the focus of intense investigation, which has led to a vast number of classification results and applications in integral geometry \cite{AleskerDescriptiontranslationinvariant2001,AleskerEtAlHarmonicanalysistranslation2011,AleskerFaifmanConvexvaluationsinvariant2014,BernigFuHermitianintegralgeometry2011,BernigEtAlIntegralgeometrycomplex2014,BoeroeczkyLudwigMinkowskivaluationslattice2019,HaberlParapatitsMomentsvaluations2016,HaberlParapatitsCentroaffinetensor2017,KlainshortproofHadwigers1995,LudwigEllipsoidsmatrixvalued2003,LudwigReitznerclassification$SLn$invariant2010,LudwigSilversteinTensorvaluationslattice2017,WannererIntegralgeometryunitary2014,Wannerermoduleunitarilyinvariant2014}. Similarly, there exists a large body of research concerning valuations on functions \cite{BaryshnikovEtAlHadwigersTheoremdefinable2013,ColesantiEtAlhomogeneousdecompositiontheorem2020,ColesantiEtAlclassinvariantvaluations2020,ColesantiEtAlContinuousvaluationsspace2021,HofstaetterKnoerrEquivariantEndomorphismsConvex2022,KnoerrSmoothvaluationsconvex2020,KnoerrUlivellivaluationsconvexbodies2023,LudwigFisherinformationmatrix2011,LudwigValuationsSobolevspaces2012,LudwigCovariancematricesvaluations2013,MussnigVolumepolarvolume2019,TradaceteVillanuevaValuationsBanachlattices2020,TsangMinkowskivaluations$Lp$2012}, which includes classifications of many important geometric operators. \\
	Monge-Amp\`ere-type operators have found a number of applications in modern valuation theory, for example in the construction of continuous valuations on convex bodies and functions  \cite{AleskerValuationsconvexsets2005,AleskerValuationsconvexfunctions2019,Kazarnovskiizerosexponentialsums1981,KazarnovskiiNewtonpolyhedraroots1984}. Similarly, the \emph{Hessian measures} play a prominent role in the characterization of a class of rotation invariant valuations on $\Conv(\R^n,\R)$ by Colesanti, Ludwig and Mussnig \cite{ColesantiEtAlHadwigertheoremconvex2020,ColesantiEtAlHadwigertheoremconvex2021,ColesantiEtAlHadwigertheoremconvex2022a,ColesantiEtAlHadwigertheoremconvex2022}, which mirrors a classical result by Hadwiger \cite{HadwigerVorlesungenuberInhalt2013}. These measures were considered previously by Trudinger and Wang \cite{TrudingerWangHessianmeasuresI1997,TrudingerWangHessianmeasuresII1999} and they are intimately related to support measures of convex bodies and singular sets of semi-convex functions \cite{ColesantiHugSteinertypeformulae2000,ColesantiHugHessianmeasuresconvex2005}. \\

	The main results of this article provide a complete description of two classes of measure-valued valuations on $\Conv(\R^n,\R)$ that share many properties with the real Monge-Amp\`ere operator. Let $U\subset\R^n$ be an open convex set. We call a map $\Psi:\Conv(U,\R)\rightarrow\mathcal{M}(U)$
\begin{itemize}
	\item \emph{locally determined} if the following holds: If $f,h\in\Conv(U,\R)$ satisfy $f\equiv h$ on an open subset $V\subset U$, then 
	\begin{align*}
		\Psi(f)[B\cap V]=\Psi(h)[B\cap V]\quad\text{for all relatively compact Borel sets } B\subset U;
	\end{align*}
	\item \emph{translation equivariant} if $U=\R^n$ and for all $f\in \Conv(\R^n,\R)$ and all $x\in\R^n$,
		\begin{align*}
			\Psi(f(\cdot+x))[B]=\Psi(f)[B+x]\quad\text{for all bounded Borel sets } B\subset\R^n;
		\end{align*}
	\item \emph{dually epi-translation invariant} if
	\begin{align*}
		\Psi(f+l)=\Psi(f)\quad\text{for all }f\in\Conv(U,\R),~ l:\R^n\rightarrow\R \text{ affine};
	\end{align*}
	\item \emph{dually simple} if $\Psi(\pi_E^*f|_U)=0$ for all $f\in\Conv(E,\R)$ and all proper subspaces $E\subset\R^n$, where $\pi_E:\R^n\rightarrow E$ denotes the orthogonal projection.
\end{itemize}
	The terminology for the last two notions stems from a geometric interpretation of these properties in terms of valuations on epi-graphs. We refer to \cite{ColesantiEtAlhomogeneousdecompositiontheorem2020,ColesantiEtAlHadwigertheoremconvex2022} for details and only note that the real Monge-Amp\`ere operator $\MA$ has all of these properties. Our first result shows that it is essentially the only continuous valuation $\Psi:\Conv(\R^n,\R)\rightarrow\mathcal{M}(\R^n)$ to do so.
	\begin{theorem}
		\label{maintheorem:characRealMA}
		Let $\Psi: \Conv(U,\R)\rightarrow\mathcal{M}(U)$ be a continuous, dually epi-translation invariant valuation. If $\Psi$ is dually simple and locally determined, then there exists a unique continuous function $\psi\in C(U)$ such that for every $f\in\Conv(U,\R)$,
		\begin{align*}
			\Psi(f)[B]=\int_B \psi(x)d\MA(f)[x]\quad\text{for all relatively compact Borel sets}~B\subset U.
		\end{align*}
	\end{theorem}
	If we assume that $U=\R^n$ and that the valuation $\Psi$ in Theorem~\ref{maintheorem:characRealMA} is in addition translation equivariant, this result implies that $\Psi$ is a scalar multiple of the real Monge-Amp\`ere operator. Let us remark that a similar result holds for real-valued, dually epi-translation invariant, continuous valuations, as shown by Colesanti, Ludwig and Mussnig \cite{ColesantiEtAlhomogeneousdecompositiontheorem2020,ColesantiEtAlHadwigertheoremconvex2022}.\\
	
	For our second main result, we drop the assumption that our valuations are dually simple and instead consider translation equivariant valuations. 
	\begin{definition}
		Let $\MAVal(\R^n)$ denote the space of all continuous, dually epi-translation invariant valuations on $\Conv(\R^n,\R)$ with values in $\mathcal{M}(\R^n)$ that are locally determined and translation equivariant.
	\end{definition}
	We will informally consider the space $\MAVal(\R^n)$ as the space of translation equivariant Monge-Amp\`ere operators on $\Conv(\R^n,\R)$, a notion that will be justified by the results below.\\

	As a consequence of a very general result about dually epi-translation invariant valuations established in \cite{Knoerrsupportduallyepi2021}, this space admits a homogeneous decomposition, similar to a classical result by McMullen \cite{McMullenValuationsEulertype1977}: Let us call $\Psi\in\MAVal(\R^n)$ \emph{$k$-homogeneous} if $\Psi(tf)=t^k\Psi(f)$ for all $f\in\Conv(\R^n,\R)$, $t\ge 0$. If we denote the corresponding subspace by $\MAVal_k(\R^n)$, then
	\begin{align*}
		\MAVal(\R^n)=\bigoplus_{k=0}^n\MAVal_k(\R^n)
	\end{align*}
	compare Section \ref{section:homogeneousDecomposition}.

	Before we state the classification result for this space of valuations, let us discuss some examples. Let us start with valuations obtained from the real Monge-Amp\`ere operator by polarization. Define the \emph{mixed Monge-Amp\`ere measure} by
	\begin{align*}
		\MA(f_1,\dots,f_n):=\frac{1}{n!}\frac{\partial^{n}}{\partial\lambda_1\dots\partial\lambda_{n}}\Big|_0\MA\left(\sum_{i=1}^{n}\lambda_if_i\right)\quad\text{for}~f_1,\dots,f_n\in\Conv(\R^n,\R).
	\end{align*}
	It is easy to see that $(\lambda_1,\dots,\lambda_k)\mapsto \MA\left(\sum_{i=1}^{n}\lambda_if_i\right)$ is a polynomial in $\lambda_1,\dots,\lambda_n\ge 0$, so this is well defined. Taking $f$ $k$-times in this expression, we obtain \emph{mixed Monge-Amp\`ere operators} of degree $k$, 
	\begin{align*}
		f\mapsto \MA(f[k],f_1,\dots,f_{n-k}),
	\end{align*}
	which define elements of $\MAVal_k(\R^n)$ if $f_1,\dots,f_{n-k}$ are quadratic polynomials. In this case, we call $f\mapsto \MA(f[k],f_1,\dots,f_{n-k})$ a  mixed Monge-Amp\`ere operator of \emph{quadratic-type}. \\

	A second way to construct functionals in this class is the following: If $f\in C^2(\R^n)$, then the graph of its differential $\mathrm{graph}(df)$ defines an oriented $n$-dimensional $C^1$-submanifold of the cotangent bundle $T^*\R^n=\R^n\times(\R^n)^*$. In particular, we can consider this graph as an $n$-dimensional current by integrating suitable differential forms. As shown by Fu \cite{FuMongeAmpereFunctions1989}, this construction extends naturally to a much larger class of functions, called \emph{Monge-Amp\`ere functions}, which includes the space $\Conv(\R^n,\R)$. To any such function $f$ one can associate an integral current $D(f)$ on $T^*\R^n$ that is uniquely characterized by a list of natural properties satisfied by the graph of the differential of a $C^2$-function. The current $D(f)$ is called the \emph{differential cycle} of $f$ and can be used to construct valuations on convex bodies \cite{AleskerFuTheoryvaluationsmanifolds2008} and convex functions \cite{KnoerrSmoothvaluationsconvex2020,KnoerrUnitarilyinvariantvaluations2021,KnoerrSingularvaluationsHadwiger2022}.\\
	
	Let $\Lambda^{n-k,k}:=\Lambda^{n-k}\R^n\otimes  \Lambda^{k}(\R^n)^*$ denote the space of constant differential forms on $T^*\R^n= \R^n\times (\R^n)^*$ of bidegree $(n-k,k)$ with complex coefficients. If $\pi:T^*\R^n\rightarrow\R^n$ denotes the projection onto the first factor, then any $\tau\in \Lambda^{n-k,k}$ defines an element $\Psi_\tau\in \MAVal_k(\R^n)$ by setting
	\begin{align*}
		\Psi_\tau(f)[B]:=D(f)[1_{\pi^{-1}(B)}\tau]\quad\text{for all bounded Borel sets }B\subset\R^n
	\end{align*}
	for $f\in\Conv(\R^n,\R)$, see Theorem~\ref{theorem:measuresDifferentialCycle} below. For $f\in\Conv(\R^n,\R)\cap C^2(\R^n)$, this measure is thus given by
	\begin{align*}
	\Psi_\tau(f)[B]=\int_{\pi^{-1}(B)\cap \mathrm{graph}(df)}\tau\quad\text{for all bounded Borel sets }B\subset\R^n.
	\end{align*}

	Our second main result provides a complete characterization of $\MAVal_k(\R^n)$ in terms of these operators.
	\begin{theorem}
		\label{maintherorem:RepresentationDifferentialForms}
		For a continuous map $\Psi:\Conv(\R^n,\R)\rightarrow\mathcal{M}(\R^n)$ the following are equivalent:
		\begin{enumerate}
			\item $\Psi\in\MAVal_k(\R^n)$.
			\item $\Psi$ is a linear combination of mixed Monge-Amp\`ere operators of quadratic type of degree $k$.
			\item There exists a differential form $\tau\in\Lambda^{n-k,k}$ such that for all $f\in\Conv(\R^n,\R)\cap C^2(\R^n)$,
			\begin{align*}
				\Psi(f)[B]=\int_{\pi^{-1}(B)\cap \mathrm{graph}(df)}\tau\quad\text{for all bounded Borel subsets}~B\subset\R^n.
			\end{align*}
			\item There exists a polynomial $P$ on the space of symmetric $(n\times n)$-matrices that is a linear combination of the $(k\times k)$-minors such that for all $f\in\Conv(\R^n,\R)\cap C^2(\R^n)$,
				\begin{align*}
					\Psi(f)[B]=\int_B P(D^2f(x))dx\quad\text{for all bounded Borel sets}~B\subset \R^n.
				\end{align*}
		\end{enumerate}
	\end{theorem}

	The proof of this result is based on the observation that all of these constructions can be related to the same class of complex polynomials - namely the space spanned by $(k\times k)$-minors of complex symmetric $(n\times n)$-matrices. This reduces Theorem~\ref{maintherorem:RepresentationDifferentialForms} to a simple representation theoretic statement about these polynomials considered as a representation of $\GL(n,\C)$. In fact, this approach leads to the following stronger characterization result: Consider $\MAVal(\R^n)$ as a representation of $\GL(n,\R)$, where the operation is given by
	\begin{align}
		\label{eq:defGlnROperation}
		[g\cdot\Psi](f)[B]=\Psi(f\circ g,g^{-1}B)\quad\text{for}~f\in\Conv(\R^n,\R),~ B\subset\R^n~\text{bounded Borel set}.
	\end{align}
	\begin{theorem}
		\label{maintheorem:IrreducibilityMAVal}
		$\MAVal_k(\R^n)$ is a finite dimensional irreducible representation of $\GL(n,\R)$.
	\end{theorem}
	Note that this directly implies that the three classes of operators given in Theorem~\ref{maintherorem:RepresentationDifferentialForms} span $\MAVal_k(\R^n)$, as they all define non-trivial $\GL(n,\R)$-invariant subspaces. \\	
	
	Let us remark that Alesker \cite{AleskerDescriptiontranslationinvariant2001} has given a similar characterization for the space of translation invariant, continuous valuations on the space of convex bodies, which is an infinite dimensional representation of the general linear group. The result above is not based on his characterization, however, both proofs rely on similar embeddings to reduce the problem to a purely representation theoretic statement.\\

	As an application of these results, we obtain a Hadwiger-type characterization of all elements $ \Psi\in\MAVal_k(\R^n)$ that are equivariant with respect to rotations, that is, that satisfy 
	\begin{align*}
		g\cdot\Psi=\Psi\quad\text{for}~g\in \SO(n),
	\end{align*}
	where we consider the operation of $\GL(n,\R)$ defined in \eqref{eq:defGlnROperation}.
	\begin{theorem}\label{maintheorem:HadwigerSOn}
		If $\Psi\in \MAVal_k(\R^n)$ is $\SO(n)$-equivariant, then $\Psi$ is a multiple of the $k$th Hessian measure, that is, there exists $c\in\C$ such that for $f\in\Conv(\R^n,\R)\cap C^2(\R^n)$,
		\begin{align*}
			\Psi(f)[B]=c\int_{B}[D^2f(x)]_kd\vol_n(x)\quad\text{for all bounded Borel sets}~B\subset\R^n,
		\end{align*}
		where $[D^2f(x)]_k$ denotes the $k$th elementary symmetric polynomial in the eigenvalues of $D^2f(x)$.
	\end{theorem} 
	We use this classification to obtain a description of elements in $\MAVal(\R^n)$ that are equivariant with respect to $\GL(n,\R)$ in Section \ref{section:HadwigerTypeResult}, which gives an alternative characterization of the real Monge-Amp\`ere operator. In Section \ref{section:positivity} we also characterize the cone of non-negative valuations, that is, all $\Psi\in \MAVal(\R^n)$ with the property that $\Psi(f)$ defines a non-negative measure for all $f\in\Conv(\R^n,\R)$.

	\subsection{Plan of the article}
	In Section \ref{section:preliminaries} we recall some notation from convex geometry and discuss some properties of the differential cycle. Section \ref{section:DuallyEpiValuations} contains the relevant background on dually epi-translation invariant valuations on convex functions and adds minor generalizations of results from \cite{ColesantiEtAlhomogeneousdecompositiontheorem2020,Knoerrsupportduallyepi2021}. In particular, we consider the Goodey-Weil distributions associated to homogeneous valuations and the Fourier-Laplace transform of these distributions. In Section \ref{section:homogeneousDecomposition} we apply these results to measure-valued valuations and prove Theorem~\ref{maintheorem:characRealMA}. We also show how locally determined valuations can be constructed using the differential cycle.\\
	Section \ref{section:Restrictions} examines a notion of restriction for elements in $\MAVal_k(\R^n)$ to $k$-dimensional subspaces of $\R^n$, which we use to prove Theorem~\ref{maintheorem:HadwigerSOn}.\\
	Theorem~\ref{maintherorem:RepresentationDifferentialForms} and Theorem~\ref{maintheorem:IrreducibilityMAVal} are finally proved in Section \ref{section:characterization}. We first introduce two families of polynomials derived from primitive differential forms and show that these polynomials coincide with certain spaces generated by $(k\times k)$-minors using some basic representation theoretic tools. We then consider the Fourier-Laplace transform of the Goodey-Weil distributions of certain complex-valued valuations obtained from elements in $\MAVal_k(\R^n)$, and we show that the resulting holomorphic functions have a characteristic structure that involves the same families of polynomials.\\
	As an application, we examine the cone of non-negative valuations in Section \ref{section:positivity}.

\section{Preliminaries and notation}
	\label{section:preliminaries}
	We consider $\R^n$ as a Euclidean vector space equipped with its standard inner product $\langle\cdot,\cdot \rangle$ and induced norm $|\cdot|$. We denote the Lebesgue measure on $\R^n$ by $\vol_n$ and the volume of the unit ball $B_1(0)$ by $\omega_n$. The Grassmanian of all $k$-dimensional linear subspaces on $\R^n$ is denoted by $\Gr_k(\R^n)$. For $E\in\Gr_k(\R^n)$ we denote the Lebesgue measure on $E$ induced by the restriction of the scalar product by $\vol_E$.\\
	
	For a locally convex vector space $F$, that is, a complex vector space equipped with a Hausdorff topology induced by semi-norms, we denote the space of all continuous linear functionals on $F$ by $F'$. Note that $F'$ separates points in $F$, that is, $\lambda(v)=0$ for all $\lambda\in F'$ implies $v=0$ for $v\in F$.

	\subsection{Convex functions}
	\label{section:ConvexFunctions}
	We refer to the monographs by Rockafellar and Wets \cite{RockafellarWetsVariationalanalysis2009} and Schneider \cite{SchneiderConvexbodiesBrunn2014} for a comprehensive background on convex functions and convex bodies. We will only need the following simple remarks on the topology of the spaces of convex functions considered in this article:\\
	If $U\subset\R^n$ is open and convex, then we equip the space $\Conv(U,\R)$ of all convex functions $f:U\rightarrow \R$ with the topology of uniform convergence on compact subsets. In particular, $\Conv(U,\R)$ is a metrizable space. This topology coincides with the topology induced by pointwise convergence, compare \cite[Theorem~7.17]{RockafellarConvexanalysis1970}, and if $U=\R^n$, then this topology also coincides with the topology induced by epi-convergence. We omit the definition of epi-convergence as we will not use it in this article.
	
	\subsection{Translation invariant valuations on convex bodies}
	\label{section:PreliminariesVal}
	Let $\mathcal{K}(\R^n)$ denote the space of \emph{convex bodies}, that is, the set of all non-empty, compact, and convex subsets of $\R^n$ equipped with the Hausdorff metric.
	Given a convex body $K\in\mathcal{K}(\R^n)$ its \emph{support function} $h_K\in \Conv(\R^n,\R)$ is defined by
	\begin{align*}
		h_K(y)=\sup_{x\in K}\langle y,x\rangle \quad\text{for}~y\in\R^n.
	\end{align*}
	The support function determines $K$ uniquely. Furthermore it has the following well known properties:
	\begin{itemize}
		\item $h_{tK}=th_K$ for all $t\ge 0$, $K\in\mathcal{K}(\R^n)$.
		\item If $K,L$ are convex bodies such that $K\cup L$ is convex, then $h_{K\cup L}=h_K\vee h_L$ and  $h_{K\cap L}=h_K\wedge h_L$.
		\item $h_{K+x}(y)=h_K(y)+\langle y,x\rangle$ for all $x,y\in\R^n$, $K\in\mathcal{K}(\R^n)$.
		\item A sequence  $(K_j)_j$ of convex bodies converges to $K$ with respect to the Hausdorff metric if and only if $(h_{K_j})_j$ converges to $h_K$ uniformly on compact subsets.
	\end{itemize}
	A map $\mu:\mathcal{K}(\R^n)\rightarrow (G,+)$ into an Abelian semi-group is called a valuation if
	\begin{align*}
		\mu(K\cup L)+\mu(K\cap L)=\mu(K)+\mu(L)
	\end{align*}
	for all $K,L\in\mathcal{K}(\R^n)$ such that $K\cup L\in\mathcal{K}(\R^n)$. For a Hausdorff topological vector space $F$, let us denote by $\Val(\R^n,F)$ the space of all valuations $\mu:\mathcal{K}(\R^n)\rightarrow F$ that are 
	\begin{itemize}
		\item translation invariant, that is $\mu(K+x)=\mu(K)$ for all $K\in\mathcal{K}(\R^n)$, $x\in\R^n$,
		\item continuous with respect to the Hausdorff metric.
	\end{itemize}
	Let $\Val_k(\R^n,F)$ denote the subspace of all valuations $\mu\in\Val(\R^n,F)$ that are \emph{$k$-homogeneous}, that is, that satisfy $\mu(tK)=t^k\mu(K)$ for all $K\in\mathcal{K}(\R^n)$, $t\ge 0$. Then the following holds.
	\begin{theorem}[McMullen \cite{McMullenValuationsEulertype1977}]
		\label{theorem:McMullenDecomp}
		$\Val(\R^n,F)=\bigoplus_{k=0}^n\Val_k(\R^n,F)$
	\end{theorem}
	For $F=\C$ we will also denote these spaces by $\Val(\R^n):=\Val(\R^n,\C)$, $\Val_k(\R^n):=\Val_k(\R^n,\C)$. These spaces carry a natural topology: For a convex body $B\in\mathcal{K}(\R^n)$ with non-empty interior consider the set $\tilde{B}:=\{K\in\mathcal{K}(\R^n):K\subset B\}$. By the Blaschke selection theorem, $\tilde{B}$ is a compact subset of $\mathcal{K}(\R^n)$. For an open set $O\subset F$, consider the set
	\begin{align*}
		\mathcal{M}(\tilde{B},O):=\{\mu\in\Val(\R^n,F):\mu(K)\in O\quad\text{for all}~K\in\tilde{B}\}.
	\end{align*}
	These subsets form a basis for the compact-open topology on $\Val(\R^n,F)$ that does not depend on the choice of $B\in\mathcal{K}(\R^n)$. Note that this topology is Hausdorff due to Theorem~\ref{theorem:McMullenDecomp}.
		
	\subsection{Monge-Amp\`ere functions and the differential cycle}
		Monge-Amp\`ere functions were introduced by Fu in \cite{FuMongeAmpereFunctions1989}. We also refer to \cite{JerrardSomerigidityresults2010} for an extension of these results. Abusing notation, we denote by $\vol_n\in\Omega^n(\R^n)$ the volume form induced by the scalar product. 
			\begin{theorem}[Fu \cite{FuMongeAmpereFunctions1989} Theorem~2.0]
			\label{theorem:characterization_differential_cycle}
			Let $f:\R^n\rightarrow\mathbb{R}$ be a locally Lipschitzian function. There exists at most one integral $n$-current $S$ on $T^*\R^n$ such that
			\begin{enumerate}
				\item $S$ is closed, i.e. $\partial S=0$,
				\item $S$ is Lagrangian, i.e. $S\llcorner \omega_s=0$, where $\omega_s$ denotes the natural symplectic form,
				\item $S$ is locally vertically bounded, i.e. $\supp S\cap \pi^{-1}(A)$ is compact for all $A\subset \R^n$ compact,
				\item $S(\phi(x,y)\pi^*\vol)=\int_{\R^n}\phi(x,df(x))d\vol(x)$ for all $\phi\in C^\infty_c(T^*\R^n)$.
			\end{enumerate}
			Note that the right hand side of the last equation is well defined due to Rademacher's theorem.
		\end{theorem}
		If such a current exists, the function $f$ is called Monge-Amp\`ere. The corresponding current is denoted by $D(f)$ (it is denoted by $[df]$ in \cite{FuMongeAmpereFunctions1989}) and is called the \emph{differential cycle of $f$}.  For the purposes of this article, we will only need the fact that any element of $\Conv(\R^n,\R)$ admits a differential cycle, compare \cite[Proposition~3.1]{FuMongeAmpereFunctions1989}\\
	
		Let us state the following two properties of the differential cycle. The first is a special case of \cite[Proposition~2.4]{FuMongeAmpereFunctions1989}.
	\begin{proposition}
		\label{proposition:differentialCycleSum}
		Let $f\in\Conv(\R^n,\R)$. If $h\in\Conv(\R^n,\R)\cap C^2(\R^n)$, then the differential cycle of $f+h$ is given by 
		\begin{align*}
			F(f+h)=G_{h*}D(f),
		\end{align*}
		where $G_h:T^*\R^n\rightarrow T^*\R^n$ is given by $(x,y)\mapsto (x,y+dh(x))$.
	\end{proposition}
	The following result is shown in \cite[Proposition~2.5]{FuMongeAmpereFunctions1989} for orientation preserving diffeomorphisms. For orientation preserving diffeomorphisms the same reasoning can be applied, compare \cite[Proposition~4.4]{KnoerrSmoothvaluationsconvex2020}.
	\begin{proposition}
		\label{proposition:differential_cycle_and_diffeomorphisms}
		Let $\phi:\R^n\rightarrow \R^n$ be a diffeomorphism of class $C^{1,1}$ and $f:\R^n\rightarrow \R$ a Monge-Amp\`ere function. Then $f\circ\phi$ is Monge-Amp\`ere and
		\begin{align*}
			D(f\circ\phi)=\left(\phi^\#\right)_*D(f),	
		\end{align*}
		if $\phi$ is orientation preserving and
		\begin{align*}
			D(f\circ\phi)=-\left(\phi^\#\right)_*D(f),
		\end{align*}
		if $\phi$ is orientation reversing. Here $\phi^\#:T^*\R^n\rightarrow T^*\R^n$ is given by $(x,y)\mapsto (\phi^{-1}(x),\phi^*y)$.
	\end{proposition} 
\section{Dually epi-translation invariant valuations}
	\label{section:DuallyEpiValuations}
 \subsection{Homogeneous decomposition and Goodey-Weil distributions}	
	Let $F$ be a locally convex vector space, $U\subset\R^n$ an open and convex set, and let $\VConv(U,F)$ denote the space of all continuous valuations $\mu$ on $\Conv(U,\R)$ with values in $F$ that are \emph{dually epi-translation invariant}, that is, that satisfy
	\begin{align*}
		\mu(f+l)=\mu(f)\quad \text{for all}~f\in\Conv(U,\R),~l:\R^n\rightarrow\R ~\text{affine}.
	\end{align*}
	In most cases, it is enough to consider the case $U=\R^n$ due to the following simple result.
	\begin{lemma}[\cite{Knoerrsupportduallyepi2021} Lemma~7.7]
		\label{lemma:InjectivityRestriction}
		The map
		\begin{align*}
			\res^*:\VConv(U,F)&\rightarrow\VConv(\R^n,F)\\
			\mu&\mapsto \left[f\mapsto \mu(f|_U)\right]
		\end{align*}
		is well defined and injective.
	\end{lemma}
	Consider the subspace $\VConv_k(U,F)$ of all \emph{$k$-homogeneous} valuations, that is, all $\mu\in\VConv(U,F)$ such that $\mu(tf)=t^k\mu(f)$ for all $t\ge 0$, $f\in\Conv(U,\R)$. 	For $F=\C$ we will also denote these spaces by $\VConv(U)$ and $\VConv_k(U)$. 
	We have the following decomposition, which for $U=\R^n$ was also obtained by Colesanti, Ludwig and Mussnig \cite{ColesantiEtAlhomogeneousdecompositiontheorem2020} in the scalar-valued case.
	\begin{theorem}
		\label{theorem:homogeneousDecompVConv}
		\begin{align*}
			\VConv(U,F)=\bigoplus_{k=0}^n\VConv_k(U,F)
		\end{align*}
	\end{theorem}	
	\begin{proof}
		This is essentially contained in Section 7.2 of \cite{Knoerrsupportduallyepi2021}, but it is explicitly stated only for the case $U=\R^n$ in \cite[Theorem~1]{Knoerrsupportduallyepi2021}. Let us deduce the statement above from this special case. Let $\mu\in\VConv(U,F)$ and define $\mu^t\in\VConv(U,F)$ for $t\ge0$ by  $\mu^t(f):=\mu(tf)$ for $f\in\Conv(U,\R)$. Then $[\res^*\mu^t](f)=[\res^*\mu](tf)$.
		By \cite[Theorem~1]{Knoerrsupportduallyepi2021} there thus exist valuations $\mu_k\in \VConv_k(\R^n,F)$ such that $\res^*\mu=\sum_{k=0}^n\mu_k$. In particular,
		\begin{align*}
			[\res^*\mu_t](f)=\sum_{k=0}^{n}t^k\mu_k(f)\quad\text{for}~f\in\Conv(\R^n,\R),t\ge 0.
		\end{align*}
		Plugging in $t=1,\dots,n+1$ and using the inverse of the Vandermonde matrix, we obtain $c_{kj}\in\R$ with
		\begin{align*}
			\mu_k(f)=\sum_{j=1}^{n+1}c_{kj}[\res^*\mu^j](f)=\res^*\left[\sum_{j=1}^{n+1}c_{kj}\mu(j\cdot)\right](f).
		\end{align*}
		If we define $\tilde{\mu}_k\in\VConv(U,F)$ by $\tilde{\mu}:=\sum_{j=1}^{n+1}c_{kj}\mu(j\cdot)$, we thus obtain
		\begin{align*}
			\res^*\tilde{\mu}_k^t=\mu_k(t\cdot)=t^k\mu_k=\res^*(t^k\tilde{\mu}_k)\quad\text{for}~t\ge0,
		\end{align*}
		so $\tilde{\mu}_k(t\cdot)=t^k\tilde{\mu}_k$ for $t\ge0$ due to the injectivity of $\res^*$. Thus $\tilde{\mu}_k\in \VConv_k(U,F)$. Similarly,
		\begin{align*}
			\res^*\left[\sum_{k=0}^{n}\tilde{\mu}_k\right](f)=\sum_{k=0}^{n}[\res^*\tilde{\mu}_k](f)=\sum_{k=0}^{n}\mu_k(f)=[\res^*\mu](f).
		\end{align*}
		As $\res^*$ is injective, this implies $\mu=\sum_{k=0}^{n}\tilde{\mu}_k$, which completes the proof.
		\end{proof}

	For the rest of this section, we will only be concerned with the case $U=\R^n$. Note that for $\mu\in\VConv_k(\R^n,F)$, the polarization $\bar{\mu}:\Conv(\R^n,\R)^k\rightarrow F$, 
	\begin{align*}
		\bar{\mu}(f_1,\dots,f_k):=&\frac{1}{k!}\frac{\partial^k}{\partial \lambda_1\dots\partial\lambda_k}\Big|_0\mu\left(\sum_{i=1}^k\lambda_i f_i\right),
	\end{align*}
	is well defined because $(\lambda_1,\dots,\lambda_k)\mapsto \mu\left(\sum_{i=1}^k\lambda_i f_i\right)$ is a polynomial in $\lambda_1,\dots,\lambda_k>0$ due to Theorem \ref{theorem:homogeneousDecompVConv}. Moreover, $\mu$ is essentially a multilinear functional that satisfies $\bar{\mu}(f,\dots,f)=\mu(f)$. As shown in \cite{KnoerrSingularvaluationsHadwiger2022}, the polarization lifts to a distribution on $(\R^n)^k$. For complex-valued valuations, these distributions may be characterized in the following way.
	\begin{theorem}[\cite{Knoerrsupportduallyepi2021} Theorem~2]
		\label{theorem:GW}
			For every $\mu\in\VConv_k(\R^n)$ there exists a unique symmetric distribution $\GW(\mu)$ on $(\R^n)^k$ with compact support which satisfies the following property: If $f_1,...,f_k\in \Conv(\R^n,\R)\cap C^\infty(\R^n)$, then
		\begin{align*}
		\GW(\mu)[f_1\otimes...\otimes f_k]=\bar{\mu}(f_1,...,f_k).
		\end{align*}
		In particular, $\mu$ is uniquely determined by $\GW(\mu)$.
	\end{theorem}
	A similar construction was used by Goodey and Weil \cite{GoodeyWeilDistributionsvaluations1984} in the context of continuous translation invariant valuations on convex bodies, and we call $\GW(\mu)$ the Goodey-Weil distribution of $\mu\in\VConv_k(\R^n)$. One may also define these distributions for valuations with values in an arbitrary locally convex vector space, however, these distributions do not have compact support in general. We refer to \cite{Knoerrsupportduallyepi2021} for more details.

	\subsection{Fourier-Laplace transform and dually simple valuations}
		In this section we obtain a description of dually simple valuations in $\VConv(U,F)$. In fact, the results provide an alternative characterization of $k$-homogeneous valuations. Let us say that $\mu\in\VConv(U,F)$ vanishes on a $k$-dimensional subspace $E\in\Gr_k(\R^n)$ if $ \mu(\pi_E^*f|_U)=0$ for all $f\in\Conv(E,\R)$. We will consider the valuation $f\mapsto \mu(\pi_E^*f|_U)$ as the restriction of $\mu\in\VConv(U,F)$ to the subspace $E$. Note that this defines an element of $\VConv(E,F)$.
		\begin{theorem}
			\label{theorem:simpleValuations}
			If $\mu\in\VConv_k(U,F)$ vanishes on all $k$-dimensional subspaces in $\R^n$, then $\mu=0$. In particular, $\mu\in\VConv(U,F)$ is dually simple if and only if $\mu$ is $n$-homogeneous.
		\end{theorem}
		\begin{proof}
			Due to Lemma~\ref{lemma:InjectivityRestriction} it is enough to consider the case $U=\R^n$. For complex-valued valuations, this was shown in \cite[Theorem~3]{KnoerrUnitarilyinvariantvaluations2021}. For the general case, note that $\mu\in \VConv_k(\R^n,F)$ vanishes identically if and only if the complex-valued valuation $\lambda\circ \mu\in\VConv(\R^n)$ vanishes for all $\lambda\in F'$, where $F'$ denotes the topological dual space of $F$. If $\mu$ vanishes on all $k$-dimensional subspaces, then the same holds true for $\lambda\circ\mu$ for $\lambda\in F'$. \cite[Theorem~3]{KnoerrUnitarilyinvariantvaluations2021} thus implies $\lambda\circ\mu=0$ for all $\lambda\in F'$. As $F'$ separates points in $F$, we obtain $\mu=0$, which shows the first claim.\\
			
			For the second claim, note that the restriction of any valuation in $\VConv_k(\R^n,F)$ to a subspace of dimension smaller than $k$ vanishes due to Theorem~\ref{theorem:homogeneousDecompVConv}. If $\mu\in\VConv(\R^n)$ is a dually simple valuation and $\mu=\sum_{k=0}^{n}\mu_k$ is its homogeneous decomposition, then its restriction to a $k$-dimensional subspace $E$ is thus given by $\mu(\pi_E^*f)=\sum_{j=0}^{k}\mu_j(\pi_E^*f)$ for $f\in \Conv(E,\R)$. We can thus use induction on $k<n$ to show that $\mu_k$ vanishes on all $k$-dimensional subspaces, which implies $\mu_k=0$ for $0\le k\le n-1$. Thus $\mu=\mu_n$ is $n$-homogeneous.
		\end{proof}
		In \cite{KnoerrUnitarilyinvariantvaluations2021} the complex-valued case was obtained as a simple consequence of the properties of the Fourier-Laplace transform of the Goodey-Weil distributions. A different proof of the characterization of dually simple valuations in $\VConv(\R^n)$ was given by Colesanti, Ludwig and Mussnig in \cite{ColesantiEtAlHadwigertheoremconvex2022}.\\		
		
		As the Fourier-Laplace transform of Goodey-Weil distributions will play an important role in the proof of Theorem~\ref{maintherorem:RepresentationDifferentialForms}, let us recall the construction and some of its basic properties. Because the distribution $\GW(\mu)$ is compactly supported for $\mu\in\VConv_k(\R^n)$, its Fourier-Laplace transform $\mathcal{F}(\GW(\mu))$ defines an entire function on $(\C^n)^k$, given for $z_1,\dots,z_k\in\C^n$ by
		\begin{align*}
			\mathcal{F}(\GW(\mu))[z_1,\dots,z_k]=\GW(\mu)[\exp(i\langle z_1,\cdot\rangle)\otimes\dots\otimes\exp(i\langle z_k,\cdot\rangle)].
		\end{align*}
		If $z_1=ix_1,\dots,z_k=ix_k$ for $x_1,\dots,x_k\in \R^n$, then this function is given by evaluating $\GW(\mu)$ in convex functions defined on the real subspace $\mathrm{span}(x_1,\dots,x_k)$ of $\R^n$, that is, we are considering the restriction of a $k$-homogeneous valuation $\mu$ to a subspace of dimension at most $k$. The restrictions of these valuations are given by the following characterization of valuations of maximal degree.		
		\begin{theorem}[Colesanti-Ludwig-Mussnig \cite{ColesantiEtAlhomogeneousdecompositiontheorem2020} Theorem~5]
			\label{theorem:valuations_top_degree}
			$\mu\in\VConv_n(\R^n)$ if and only if there exists a (necessarily unique) function $\phi\in C_c(\R^n)$ such that
			\begin{align*}
			\mu(f)=\int_{\R^n}\phi d\MA(f)\quad\forall f\in\Conv(\R^n,\R).
			\end{align*}
		\end{theorem}
		This has the following implication. Let $\MA_E$ denote the real Monge-Amp\`ere operator on $E\in\Gr_k(\R^n)$.
		\begin{lemma}[\cite{KnoerrUnitarilyinvariantvaluations2021} Lemma~2.6]
			\label{lemma:interpretation_fourier_maxdegree}
			Let $\mu\in \VConv_k(\R^n)$. For $E\in \Gr_k(\R^n)$ let $\phi_E\in C_c(E)$ denote the unique function such that $\mu(\pi_{E}^*f)=\int_E\phi_Ed\MA(f)$ for all $f\in \Conv(E,\R)$. Then for $z_1,\dots,z_k\in E_\C:=E\otimes_\R\C$
			\begin{align*}
			\mathcal{F}(\GW(\mu))[z_1,\dots,z_k]=\frac{(-1)^k}{k!}\det\left(
			\langle z_i,z_j\rangle\right)_{i,j=1}^k\mathcal{F}_E(\phi_E)\left[\sum_{i=1}^{k}z_i\right].
			\end{align*}
			Here $\mathcal{F}_E$ denotes the Fourier-Laplace-Laplace transform on $L^1(E)$ and $\langle\cdot,\cdot\rangle$ denotes the $\C$-linear extension of the scalar product on $\R^n$.
		\end{lemma}
		
		The following vanishing property will be crucial for the proof of Theorem~\ref{maintherorem:RepresentationDifferentialForms}.
		\begin{lemma}
			\label{lemma:ZeroSetFourier-Laplace}
			Let $\mu\in\VConv_k(\R^n)$. $\mathcal{F}(\GW(\mu))[z_1,\dots,z_k]=0$ if $z_1,\dots,z_k\in\C^n$ are linearly dependent.
		\end{lemma}
		\begin{proof}
			If $z_1,\dots,z_k\in E_\C$ for some real subspace $E\in\Gr_k(\R^n)$, then the claim follows from Lemma~\ref{lemma:interpretation_fourier_maxdegree}. For the general case, fix $a_1,\dots,a_{k-1}\in \C$ and consider the holomorphic function $H$ on $(\C^n)^{k-1}$ given by
			\begin{align*}
			H(z_1,\dots,z_{k-1})=\mathcal{F}(\GW(\mu))\left[z_1,\dots,z_{k-1},\sum_{j=1}^{k-1}a_jz_j\right].
			\end{align*}
			Then $H$ vanishes on the totally real subspace $(\R^n)^{k-1}\subset(\C^n)^{k-1}$ and thus has to vanish identically on $(\C^n)^{k-1}$. As this holds for any choice of $a_1,\dots,a_{k-1}\in\C$ and $\mathcal{F}(\GW(\mu))$ is symmetric, we deduce that $\mathcal{F}(\GW(\mu))$ vanishes on all $(z_1,\dots,z_k)\in\C^n$ that are linearly dependent.
		\end{proof}

\section{Measure-valued valuations on convex functions}
\subsection{Homogeneous decomposition}
\label{section:homogeneousDecomposition}
\begin{proposition}
	Let $\Psi\in\VConv(U,\mathcal{M}(U))$ and $\Psi=\sum_{k=0}^{n}\Psi_k$ its homogeneous decomposition in $\VConv(U,\mathcal{M}(U))$.
	\begin{enumerate}
		\item If $\Psi$ is locally determined, then so is $\Psi_k$ for $0\le k\le n$.
		\item If $\Psi$ is translation equivariant, then  so is $\Psi_k$ for $0\le k\le n$.
	\end{enumerate}
\end{proposition}
\begin{proof}
	Note that the map $t\mapsto \Psi(tf)$ is a polynomial of degree at most $n$ with values in $\mathcal{M}(U)$ as $\Psi(tf)=\sum_{k=0}^{n}t^k\Psi_k(f)$. Using the inverse of the Vandermonde matrix, we obtain $c_{kj}\in\R$ such that
	\begin{align*}
 	\Psi_k(f)=\sum_{j=1}^{n+1}c_{kj}\Psi(jf).
 	\end{align*}
	Using this representation, it is easy to see that if $\Psi$ is either locally determined or translation equivariant, then so is $\Psi_k$ for $0\le k\le n$.
\end{proof}

Recall that $\MAVal(\R^n)$ denotes the subspace of $\VConv(\R^n,\mathcal{M}(\R^n))$ of all elements that are  locally determined and  translation equivariant.
\begin{corollary}
	$\MAVal(\R^n)=\bigoplus\limits_{k=0}^n\MAVal_k(\R^n)$
\end{corollary}
\begin{lemma}
	\label{lemma:0homMA}
		$\MAVal_0(\R^n)$ is one dimensional and spanned by the Lebesgue measure.
\end{lemma}
\begin{proof}
	As any $\Psi\in\MAVal_0(\R^n)$ is $0$-homogeneous, it is constant, that is, $\Psi(f)=\Psi(0)$ for all $f\in\Conv(\R^n,\R)$. As $\Psi$ is translation equivariant, $\Psi(0)$ is a translation invariant Radon measure on $\R^n$, so it is a multiple of the Lebesgue measure.
\end{proof}

\subsection{Dually simple and locally determined measure-valued valuations}
	Our proof of Theorem~\ref{maintheorem:characRealMA} is based on the observation by Colesanti, Ludwig and Mussnig that valuations on $\Conv(\R^n,\R)$ are completely determined by their restriction to translates of support functions. 
	\begin{lemma}
		\label{lemma:valuations-on-conv:characteristic-function-injective}
		Let $(G,+)$ be an Abelian semi-group with cancellation law that carries a Hausdorff topology, and $\mu_1,\mu_2:\Conv(\R^n,\R)\rightarrow G$ two continuous valuations. If $\mu_1(h_K(\cdot-x)+c)=\mu_2(h_K(\cdot-x)+c)$ for all $K\in\mathcal{K}(\R^n)$, $x\in \R^n$ and $c\in\R$, then $\mu_1\equiv \mu_2$ on $\Conv(\R^n,\R)$.
	\end{lemma}
	\begin{proof}
		This is a slightly weaker form of \cite{MussnigVolumepolarvolume2019} Lemma~5.1. To be precise, the version in \cite{MussnigVolumepolarvolume2019} considers translation invariant valuations, however, the proof only uses the weaker property stated above.
	\end{proof}
	Let $F$ be a locally convex vector space. For $\mu\in \VConv(\R^n,F)$ and $x\in\R^n$, define
	\begin{align*}
		S(\mu)[\cdot,x]:\mathcal{K}(\R^n)&\rightarrow F\\
		K&\mapsto S(\mu)[K,x]=\mu(h_K(\cdot-x)).
	\end{align*} 
	Let $S(\mu)$ denote the function $x\mapsto S(\mu)[\cdot,x]$ 
\begin{theorem}
	\label{theorem_embedding_Cont}
	 $S:\VConv(\R^n,F)\rightarrow C(\R^n,\Val(\R^n,F))$ is well defined and injective.
\end{theorem}
\begin{proof}
	The properties of the support function (compare Section \ref{section:PreliminariesVal}) imply that $S(\mu)[x]\in\Val(\R^n,F)$ for all $x\in U$. To see that $S$ is injective, assume that $S(\mu)=0$. As $\mu$ is invariant under the addition of constants, this implies $\mu(h_K(\cdot-x)+c)=0$ for all $K\in\mathcal{K}(\R^n)$ and all $x\in \R^n$, $c\in\R$, so Lemma~\ref{lemma:valuations-on-conv:characteristic-function-injective} implies $\mu\equiv0$ on $\Conv(\R^n,\R)$.\\
	Now let $B\subset \R^n$ be a convex body with non-empty interior and set $\tilde{B}:=\{K\in\mathcal{K}(\R^n):K\subset B\}$. Recall that a basis for the topology of $\Val(\R^n,F)$ is given by the open subsets
	\begin{align*}
	\mathcal{M}(\tilde{B},O)=\{\mu\in\Val(\R^n,F): \mu(K)\in O \ \forall K\in\tilde{B}\},
	\end{align*}
	where $O\subset F$ is open. Let us show that the map
	\begin{align*}
	R:\mathcal{K}(\R^n)\times\R^n&\rightarrow \Conv(\R^n,\R)\\
	(K,x)&\mapsto h_K(\cdot-x)
	\end{align*}
	is continuous. Assume that  $(K_j,x_j)$ converges to $(K,x)\in\mathcal{K}(\R^n)\times\R^n$. Then
	\begin{align*}
	&\left|h_{K_j}(\cdot-x_j)-h_K(\cdot-x)\right|\le \left|h_{K_j}(\cdot-x_j)-h_{K}(\cdot-x_j)\right|+\left|h_{K}(\cdot-x_j)-h_{K}(\cdot-x)\right|.
	\end{align*}
	As the sequence of functions $(h_{K_j})_j$ converges uniformly on the compact subset $\{x_j: j\in\mathbb{N}\}\cup \{x\}$, the right hand side converges pointwise to $0$. Thus the sequence $(h_{K_j}(\cdot-x_j))_j$ converges to $h_K(\cdot-x)$ in $\Conv(\R^n,\R)$.\\
	We deduce that $(K,x)\mapsto S(\mu)[x](K)=[\mu\circ R](K,x)$ is uniformly continuous on compact subsets. In particular, given an open neighborhood $O\subset F$ of the origin, we can find $\delta>0$ such that
	\begin{align*}
	\mu(h_K(\cdot-x))-\mu(h_K(\cdot-x'))\in O \quad \text{for all}~ K\subset B,~ x,x'\in \R^n\text{ with } |x-x'|<\delta,
	\end{align*} 
	that is, 
	\begin{align*}
	S(\mu)[x]-S(\mu)[x']\in \mathcal{M}(\tilde{B},O)\quad \text{for all }x,x'\in \R^n \text{ with }|x-x'|<\delta.
	\end{align*}
	Thus $S(\mu)$ is continuous.
\end{proof}
\begin{corollary}
	\label{cor:characteristicFunctionMaxDeg}
	If $\mu\in\VConv_n(\R^n,F)$, then $S(\mu)[K,x]=\frac{1}{\omega_n}S(\mu)[B_1(0),x]\vol_n(K)$ for all $K\in\mathcal{K}(\R^n)$, $x\in\R^n$.
\end{corollary}
\begin{proof}
	If $\lambda\in F'$, then $K\mapsto\lambda(S(\mu)[K,x])$ belongs to $\Val_n(\R^n)$ and is thus a constant multiple of the Lebesgue measure due to a result by Hadwiger \cite{HadwigerVorlesungenuberInhalt2013}, that is, there exists $c_\lambda(x)\in \C$ such that $\lambda(S(\mu)[K,x])=c_\lambda(x) \vol(K)$ for all $K\in \mathcal{K}(\R^n)$. For $K=B_1(0)$ we therefore obtain
	\begin{align*}
		c_\lambda(x)=\frac{1}{\omega_n}\lambda(S(\mu)[B_1(0),x]).
	\end{align*}
	Thus \begin{align*}
		\lambda\left(S(\mu)[K,x]-\frac{1}{\omega_n}S(\mu)[B_1(0),x]\vol(K)\right)=0\quad\text{for all }\lambda\in F',
	\end{align*}
	which shows the claim as $F'$ separates points in $F$.
\end{proof}
The following Lemma~is the key ingredient for the proof of Theorem~\ref{maintheorem:characRealMA}.
\begin{lemma}
	\label{lemma:supportFunctionFace}
	Let $P\subset \R^n$ be a polytope. For every $y\in \R^n\setminus\{0\}$ there a neighborhood $U$ of $y_0$ and a polytope $P'$ of dimension at most $n-1$ such that $h_{P'}=h_P$ on $U$.
\end{lemma}
\begin{proof}
	Fix $y_0\in\R^n\setminus\{0\}$. As $P$ is convex, the function $x\mapsto \langle y_0,x\rangle$ attains its maximum $h_P(y_0)$ in certain vertices of $P$. If $P$ has $m$ vertices, we may thus assume that the vertices $v_1,\dots,v_m$ of $P$ satisfy
	\begin{align*}
	h_P(y_0)=&\langle y_0,v_i\rangle \quad \text{for }1\le i\le k,\\
	h_P(y_0)>&\langle y_0,v_i\rangle \quad \text{for }k+1\le i\le m,
	\end{align*}
	for some $1\le k\le m$.
	Note that this implies that $v_1,\dots,v_k$ belong to the face $P':=\{x\in P: \langle x,y_0\rangle =h_P(y_0)\}$, which is of dimension at most $n-1$. If $k=m$, this shows that claim, as $P'=P$ in this case. Thus assume $1\le k <m$.\\
	As $h_P$ is continuous, we may choose a neighborhood $U\subset \R^n\setminus\{0\}$ of $y_0$ such that for all $y\in U$
	\begin{align*}
	h_P(y)>\langle y,v_i\rangle  \quad \text{for }k+1\le i\le m.
	\end{align*}
	For $y\in U$, the function $x\mapsto \langle y,x\rangle$ attains its maximum $h_P(y)$ on a vertex of $P$, so we find $1\le i\le k$ such that $h_P(y)=\langle v_i,y\rangle$. In particular, $h_P(y)\le h_{P'}(y)$ for all $y\in U$, which implies $h_{P'}(y)=h_P(y)$ for all $y\in U$ as $P'\subset P$.
\end{proof}

Recall that, by Theorem~\ref{theorem:simpleValuations}, $\Psi\in \VConv(U,\mathcal{M}(U))$ is dually simple if and only if it is $n$-homogenous. Theorem~\ref{maintheorem:characRealMA} thus follows from the following result. 
\begin{theorem}
	\label{theorem:characteristicFunction}
	Let $\Psi\in\VConv_n(U,\mathcal{M}(U))$ be locally determined. 
	Then there exists a unique function $\psi\in C(U)$ such that for $f\in\Conv(U,\R)$
	\begin{align}
		\label{eq:representationMeasureMaxDegree}
		\Psi(f)[B]=\int_{B}\psi(x)d\MA(f,x)\quad\text{for all relatively compact Borel sets}~B\subset U.
	\end{align}
	Moreover, $\psi$ is given by
	\begin{align*}
		\psi(x)=\frac{1}{\omega_n}S(\Psi)[B_1(0),x](\{x\})\quad\text{for}~ x\in U.
	\end{align*}
\end{theorem}
\begin{proof}
	Given $\psi\in C(U)$, the right hand side of \eqref{eq:representationMeasureMaxDegree} defines an element $\Psi_\psi\in\VConv_n(U,\mathcal{M}(U))$. As $\MA(h_K(\cdot-x))=\vol_n(K)\delta_x $ (see, for example, \cite[Lemma~2.4]{AleskerValuationsconvexfunctions2019} for a sketch of proof), we obtain
	\begin{align*}
		\psi(x)=\frac{1}{\omega_n}S(\Psi_\psi)[B_1(0),x](\{x\})=\frac{1}{\omega_n}\Psi_\psi(h_{B_1(0)}(\cdot-x)[\{x\}]\quad\text{for}~x\in U,
	\end{align*}
	so this representation is unique.\\
	
	Now let $\Psi\in \VConv_n(U,\mathcal{M}(U))$. We will consider $\Psi$ as an element of $\VConv_n(\R^n,\mathcal{M}(U))$ using Lemma~\ref{lemma:InjectivityRestriction}.	We claim that the fact that $\Psi$ is locally determined implies that there exists $c(K,x)\in \C$ for $K\in\mathcal{K}(\R^n)$, $x\in \R^n$ such that
	\begin{align*}
	S(\Psi)[K,x]=\begin{cases}
		c(K,x)\delta_x & \text{for}~x\in U,\\
		0 & \text{else}.
	\end{cases}
	\end{align*}
	Indeed, let $\phi\in C_c(U)$ be a function such that the support of $\phi$ does not contain the point $x$. If $K=P$ is a polytope, then $h_P(\cdot-x)=h_{P'}(\cdot-x)$ on a neighborhood of each point $y\in \R^n\setminus\{x\}$ for some polytope $P'$ of dimension at most $n-1$ by Lemma~\ref{lemma:supportFunctionFace}. Using a partition of unity, we can assume that the support of $\phi$ is contained in such a neighborhood. As $\Psi$ is locally determined, this implies
	\begin{align*}
	\Psi(h_P(\cdot-x))[\phi]=\Psi(h_{P'}(\cdot-x))[\phi].
	\end{align*}
	However, $\Psi$ is a dually simple valuation, so the right hand side vanishes. Thus $S(\Psi)[\cdot,x][\phi]$ vanishes on polytopes and therefore on all convex bodies by continuity. If $x\notin U$, this implies that $S(\Psi)[\cdot,x]$ vanishes identically. If $x\in U$, the measure $S(\Psi)[K,x]\in\mathcal{M}(U)$ is supported on $\{x\}$ for every $K\in\mathcal{K}(\R^n)$ and is therefore a multiple of $\delta_x$.	By Corollary~\ref{cor:characteristicFunctionMaxDeg},
	\begin{align*}
	S(\Psi)[K,x]=	\frac{1}{\omega_n}S(\Psi)[B_1(0),x]\vol(K),
	\end{align*}
	so for every $x\in U$, $S(\Psi)[K,x]=\psi(x)\vol(K)\delta_x$ for 
	\begin{align*}
		\psi(x):=\frac{1}{\omega_n}S(\mu)[B_1(0),x](\{x\})=\frac{1}{\omega_n}\Psi(h_{B_1(0)}(\cdot-x))[\{x\}].
	\end{align*}
	Note that $\psi$ is continuous on $U$: For any compact subset $K\subset U$, there exists $\phi\in C_c(U)$ which is equal to $1$ on a neighborhood of this subset. Thus
	\begin{align*}
		\psi(x)=\frac{1}{\omega_n}S(\Psi)[B_1(0),x](\{x\})=\frac{1}{\omega_n}S\left(\int_{\R^n}\phi d\Psi\right)[B_1(0),x],\quad\text{for}~x\in K.
	\end{align*}
	where $\int_{\R^n}\phi d\Psi\in \VConv_n(U)$ is a continuous, complex-valued valuation. As $S\left(\int_{\R^n}\phi d\Psi\right)$ defines a continuous function on $U$ by Theorem~\ref{theorem_embedding_Cont}, $\psi$ coincides with a continuous function on every compact subset of $U$ and is thus continuous.\\
	Next, note that the functional $\tilde{\Psi}:\Conv(\R^n,\R)\rightarrow\mathcal{M}(U)$ given for $f\in\Conv(\R^n,\R)$ by
	\begin{align*}
		\tilde{\Psi}(f)[B]=\int_{B}\psi(x)d\MA(f,x)\quad\text{for all relatively compact Borel sets}~B\subset U
	\end{align*}
	defines an element of $\VConv_n(U,\mathcal{M}(U))$ as $\psi$ is continuous. Because $\MA(h_K(\cdot-x))=\vol(K)\delta_x$, $S(\tilde{\Psi})[x]=S(\Psi)[x]$ for $x\in \R^n$. As $S$ is injective due to Theorem~\ref{theorem:characteristicFunction}, this implies $\Psi=\tilde{\Psi}$, which shows the representation formula.
\end{proof}

As a corollary, we obtain the following characterization of the real Monge-Amp\`ere operator.
\begin{corollary}
	\label{corollary:CharacterizationMA}
	Let $\Psi:\Conv(\R^n,\R)\rightarrow \mathcal{M}(\R^n)$ be a continuous, dually simple, and dually epi-translation invariant valuation. If $\Psi$ is locally determined and translation equivariant, then there exists $c\in \C$ such that $\Psi=c\cdot\MA$. In particular, $\MAVal_n(\R^n)$ is $1$-dimensional and spanned by $\MA$.
\end{corollary}

	\subsection{Construction of measure-valued valuations using the differential cycle}
	\begin{theorem}
		\label{theorem:measuresDifferentialCycle}
		Let $\tau\in\Omega^n(T^*\R^n)$ be a continuous differential $n$-form and define $\Psi_\tau(f)\in\mathcal{M}(\R^n)$ for $f\in\Conv(\R^n)$ by 
		\begin{align*}
			\Psi_\tau(f)[B]:=D(f)[1_{\pi^{-1}(B)}\tau]\quad\text{for all bounded Borel sets } B\subset\R^n.
		\end{align*}
		Then $\Psi_\tau:\Conv(\R^n,\R)\rightarrow\mathcal{M}(\R^n)$ is a continuous and locally determined valuation. If $\tau$ is invariant with respect to translations in the second factor of $T^*\R^n=\R^n\times(\R^n)^*$, then $\Psi_\tau$ is dually epi-translation invariant. If $\tau$ is invariant with respect to translations in the first factor, then $\Psi_\tau$ is translation equivariant.
	\end{theorem}
	\begin{proof}
		Note that for $\phi\in C_c(\R^n)$, $f\in\Conv(\R^n,\R)$,
		\begin{align*}
			\int_{\R^n}\phi(x)d\Psi_\tau(f,x)=D(f)[\pi^*\phi\wedge\tau].
		\end{align*}
		It was shown in \cite[Corollary~4.7]{KnoerrSingularvaluationsHadwiger2022}  that the right hand side of this equation depends continuously on $f\in\Conv(\R^n,\R)$. In particular, $\Psi_\tau$ is weakly*-continuous. It follows from the remarks in \cite[Section~2.1]{FuMongeAmpereFunctions1989} that $f=h$ on an open subset $U\subset\R^n$ implies $D(f)|_{\pi^{-1}(U)}=D(h)|_{\pi^{-1}(U)}$, which shows that $\Psi_\tau$ is locally determined.\\
		
		If $\tau$ is invariant with respect to translations in the second factor, then $1_{\pi^{-1}(B)}\wedge \tau$ has the same property for all bounded Borel sets $B\subset\R^n$. If $\lambda:\R^n\rightarrow\R$ is an affine function, then Proposition~\ref{proposition:differentialCycleSum} implies
		\begin{align*}
			\Psi_\tau(f+\lambda)[B]=&G_{\lambda*}D(f)[1_{\pi^{-1}(B)}\wedge \tau]=D(f)[G_\lambda^* (1_{\pi^{-1}(B)}\wedge \tau)]=D(f)[1_{\pi^{-1}(B)}\wedge \tau]\\
			=&\Psi_\tau(f)[B].
		\end{align*}
		Thus, $\Psi_\tau$ is dually epi-translation invariant.\\
		
		Finally, assume that $\tau$ is invariant with respect to translations in the first factor. For $x_0\in \R^n$, consider the map 
		\begin{align*}
			\Phi_{x_0}:T^*\R^n&\rightarrow T^*\R^n\\
			(x,y)&\mapsto (x-x_0,y).
		\end{align*}
		Then $D(f(\cdot+x_0))=\Phi_{x_0*}D(f)$ according to Proposition~\ref{proposition:differential_cycle_and_diffeomorphisms}, and we obtain for all $\phi\in C_c(\R^n)$
		\begin{align*}
			&\int_{\R^n}\phi(\cdot-x_0)d\Psi(f)=D(f)[\pi^*\phi(\cdot-x_0)\wedge\tau]=D(f)\left[\Phi_{x_0}^*\left(\pi^*\phi \wedge \tau\right)\right]\\
			=&\left(\Phi_{x_0*}D(f)\right)[\pi^*\phi \wedge \tau]=D(f(\cdot+x_0))[\pi^*\phi \wedge \tau]=\int_{\R^n}\phi d\Psi(f(\cdot+x_0)).
		\end{align*}
		Thus $\Psi_\tau$ is translation equivariant.
	\end{proof}
	Note that $\GL(n,\R)$ operates on $T^*\R^n=\R^n\times(\R^n)^*$ by the diagonal action
	\begin{align*}
		g^\#(x,y)=(g^{-1}x,g^*y)\quad \text{for}~g\in\GL(n,\R),~(x,y)\in \R^n\times(\R^n)^*.
	\end{align*}
	Consequently, $\GL(n,\R)$ operates on $\Lambda^{n-k,k}$ by pullbacks:
	\begin{align*}
		g\cdot \tau=(g^\#)^*\tau\quad \text{for}~g\in\GL(n,\R),~\tau\in \Lambda^{n-k,k}.
	\end{align*}
	\begin{corollary}
		\label{corollary:equivarianceMeasuresFromDifferentialForms}
		The map 
		\begin{align*}
			\Lambda^{n-k,k}&\rightarrow\MAVal_k(\R^n)\\
			\tau&\mapsto \Psi_\tau
		\end{align*}
		is well defined and $\GL(n,\R)$-equivariant in the following sense: For $f\in\Conv(\R^n,\R)$
		\begin{align*}
			\Psi_{g\cdot\tau}(f)[B]=\mathrm{sign}(\det(g) )\Psi_\tau(f\circ g)[g^{-1}B]
		\end{align*}
		for all bounded Borel sets $B\subset\R^n$.
	\end{corollary}
	\begin{proof}
		It follows from Theorem~\ref{theorem:measuresDifferentialCycle} that this map is well defined. Using Proposition~\ref{proposition:differential_cycle_and_diffeomorphisms}, we obtain for $g\in\GL(n,\R)$, $\tau\in \Lambda^{n-k,k}$,
		\begin{align*}
			\Psi_{(g^\#)^*\tau}(f)[B]=&D(f)[1_{\pi^{-1}(B)}(g^\#)^*\tau]\\
			=&D(f)[(g^\#)^*(((g^{-1)})^\#)^*1_{\pi^{-1}(B)}\tau)]\\
			=&\left((g^\#)_*D(f)\right)[1_{\pi^{-1}(g^{-1}B)}\tau]\\
			=&\mathrm{sign}(\det(g)) D(f\circ g)[1_{\pi^{-1}(g^{-1}B)}\tau]\\
			=&\mathrm{sign}(\det(g)) \Psi_\tau(f\circ g)[g^{-1}B],
		\end{align*}
		so this map is $\GL(n,\R)$-equivariant.
	\end{proof}

\section{Translation equivariant valuations and restrictions to subspaces}
\label{section:Restrictions}
For a subspace $E\subset\R^n$ let $\pi_E:\R^n\rightarrow E$ denote the orthogonal projection and $E^\perp\subset\R^n$ its orthogonal complement. Given a finite measure $\mu$ on $\R^n$, we will denote its pushforward along $\pi_{E^\perp}$ by $\pi_{E^\perp*}\mu$.
\begin{lemma}
	\label{lemma:decompRestriction}
	Fix $\Psi\in \MAVal_k(\R^n)$ and let $E\subset\R^n$ be a $k$-dimensional subspace, $f\in\Conv(E,\R)$. Then
	\begin{align*}
		\pi_{E^\perp*}[\Psi(\pi_E^*f)\llcorner \pi_E^*\phi_E]\in\mathcal{M}(E^\perp)
	\end{align*}
	is a multiple of the Lebesgue measure on $E^\perp$ for every $\phi_E\in C_c(E)$.
\end{lemma}
\begin{proof}
	Note that $\pi_{E^\perp*}[\Psi(\pi_E^*f)\llcorner \phi_E]\in \mathcal{M}(E^\perp)$ is translation invariant: If $x_{E^\perp}\in E^\perp$ and $\psi\in C_c(E^\perp)$, then
	\begin{align*}
		\int_{E^\perp} \psi(\cdot-x_{E^\perp})d\left(\pi_{E^\perp*}[\Psi(\pi_E^*f)\llcorner \pi_E^*\phi_E]\right)=&\int_{\R^n}\pi_E^*\phi_E\cdot \pi_{E^\perp}^*\psi(\cdot-x_{E^\perp})d\Psi(\pi_E^*f)\\
		=&\int_{\R^n}\pi_E^*\phi_E\cdot \pi_{E^\perp}^*\psi d\Psi(\pi_E^*f(\cdot+x_{E^\perp})),
	\end{align*}
	as $\Psi$ is translation equivariant. But $\pi_E^*f(\cdot+x_{E^\perp})=\pi_E^*f$. Therefore, this measure is translation invariant and consequently a multiple of the Lebesgue measure on $E^\perp$.
\end{proof}
\begin{proposition}
	\label{proposition:KlainEmbedding}
	For every $\Psi\in \MAVal_k(\R^n)$ there exists a unique continuous function $\KL_\Psi\in C(\Gr_k(\R^n))$ such that for $E\in\Gr_k(\R^n)$
	\begin{align*}
		\Psi(\pi_{E}^*f)=\KL_\Psi(E)\MA_E(f)\otimes \vol_{E^\perp}\quad\text{for}~f\in\Conv(E,\R).
	\end{align*}
	Moreover, $\Psi=0$ if and only if $\KL_\Psi=0$.
\end{proposition}
\begin{proof}
	Given $f\in\Conv(E,\R)$, $\pi_{E^\perp*}[\Psi(\pi_E^*f)\llcorner \phi_E]$ is a multiple of the Lebesgue measure on $E^\perp$ by Lemma~\ref{lemma:decompRestriction}. Fix $\phi\in C_c(E^\perp)$ with $\int_{E ^\perp}\phi(x)d\vol_{E^\perp}(x)=1$. Then
	\begin{align*}
		\tilde{\Psi}:\Conv(E,\R)\rightarrow&\mathcal{M}(E)=(C_c(E))'\\
		f\mapsto&\left(\phi_E\mapsto  \int_{E^\perp}\phi d[\pi_{E^\perp*}[\Psi(\pi_E^*f)\llcorner \phi_E]]\right)
	\end{align*} 
	belongs to $\MAVal_k(E)$ and is thus a multiple of $\MA_E$. Hence, there exists $\KL_\Psi(E)\in\C$ such that 
	\begin{align*}
		\tilde{\Psi}(\pi_{E}^*\cdot)=\KL_\Psi(E)\MA_E.
	\end{align*} 
	For $f\in \Conv(E,\R)$ consider the measure $\Psi_0(f):=\KL_\Psi(E)\MA_E(f)\otimes \vol_{E^\perp}\in \mathcal{M}(\R^n)$. For $\phi_E\in C_c(E)$ and $\phi_{E^\perp}\in C_c(E^\perp)$, we have
	\begin{align*}
		&\int_{\R^n}\pi_E^*\phi_E\cdot \pi_{E^\perp}^*\phi_{E^\perp}d\Psi_0(f)=\KL_\Psi(E)\int_E\phi_E d\MA_E(f)\cdot \int_{E^\perp}\phi_{E^\perp} d\vol_{E^\perp}\\
		=&\int_E\phi_Ed\tilde{\Psi}(f)\int_{E^\perp}\phi_{E^\perp} d\vol_{E^\perp}=\int_{E^\perp}\phi d[\pi_{E^\perp*}[\Psi(\pi_E^*f)\llcorner \phi_E]]\int_{E^\perp}\phi_{E^\perp} d\vol_{E^\perp}.
	\end{align*}
	As $\pi_{E^\perp*}[\Psi(\pi_E^*f)\llcorner \phi_E]]$ is a multiple of the Lebesgue measure and $\int_{E^\perp}\phi d\vol_{E^\perp}=1$,
	\begin{align*}
		\int_{E^\perp}\phi d[\pi_{E^\perp*}[\Psi(\pi_E^*f)\llcorner \phi_E]]\int_{E^\perp}\phi_{E^\perp} d\vol_{E^\perp}=\int_{E^\perp}\phi_{E^\perp} d[\pi_{E^\perp*}[\Psi(\pi_E^*f)\llcorner \phi_E]],
	\end{align*}
	so we obtain
	\begin{align*}
		\int_{\R^n}\pi_E^*\phi_E\cdot \pi_{E^\perp}^*\phi_{E^\perp}d\Psi_0(f)=&\int_{E^\perp}\phi_{E^\perp} d[\pi_{E^\perp*}[\Psi(\pi_E^*f)\llcorner \phi_E]]\\
		=&\int_{\R^n}\pi_E^*\phi_E\cdot \pi_{E^\perp}^*\phi_{E^\perp}d\Psi(\pi_E^*f).
	\end{align*}
	By the Stone-Weierstrass theorem, functions of the form $x\mapsto \phi_E(\pi_E(x))\cdot \phi_{E^\perp}(\pi_{E^\perp}(x))$ span a dense subspace of $C_c(\R^n)$. Thus 
	\begin{align*}
		\Psi(\pi_E^*f)=\Psi_0(f)=\KL_\Psi(E)\MA_E(f)\otimes \vol_{E^\perp}\quad\text{for}~f\in\Conv(E,\R).
	\end{align*}
	
	We obtain a well defined map $\KL_\Psi:\Gr_k(\R^n)\rightarrow\C$. To see that $\KL_\Psi$ is continuous, observe that we can recover this functions in the following way: First, choose $\phi_1,\phi_2\in C_c([0,\infty))$ such that $\phi_1(0)=1$, $\int_{\R^{n-k}}\phi_2(|x|)d\vol_{n-k}(x)=1$. If $B_E$ denotes the unit ball in $E$, then
	\begin{align*}
		\int_{\R^n}\pi_E^*\phi_1\cdot \pi_{E^\perp}^*\phi_2d\Psi(\pi^*_Eh_{B_E})=\KL_\Psi(E)\int_{\R^n}\pi_E^*\phi_1\cdot \pi_{E^\perp}^*\phi_2d(\MA_E(h_{B_E})\otimes \vol_{E^\perp}).
	\end{align*}
	As $\MA_E(h_{B_E})=\omega_k\delta_0$, we thus obtain
	\begin{align*}
		\KL_\Psi(E)=\frac{1}{\omega_k}\int_{\R^n}\pi_E^*\phi_1\cdot \pi_{E^\perp}^*\phi_2d\Psi(\pi^*_Eh_{B_E}),
	\end{align*}
	where the right hand side depends continuously on $E$. Thus $\KL_\Psi$ is continuous.\\
	
	Finally, let us show that $\KL_\Psi=0$ implies $\Psi=0$. Note that $\Psi=0$ if and only if the valuations
	\begin{align*}
		\mu_{\phi}(f):=\int_{\R^n}\phi d\Psi(f)
	\end{align*}
	vanish identically for all $\phi\in C_c(\R^n)$. By the characterization of dually simple valuations in Theorem~\ref{theorem:simpleValuations}, this is the case if and only if the restrictions
	\begin{align*}
		\mu_{\phi}(\pi_E^*f):=\int_{\R^n}\phi d\Psi(\pi_E^*f)=\KL_\Psi(E)\int_{\R^n}\phi d(\MA(f)\otimes\vol_{E^\perp}), \quad f\in\Conv(E,\R),
	\end{align*}
	vanishes identically for all $E\in \Gr_k(\R^n)$. Fix $E\in\Gr_k(\R^n)$. By the Stone-Weierstrass theorem, linear combinations of functions of the form $x\mapsto \phi_1(\pi_E(x))\cdot \phi_2(\pi_{E^\perp}(x))$, $\phi_1\in C_c(E)$, $\phi_2\in C_c(E^\perp)$ are dense in $C_c(\R^n)$. Thus the valuations $\mu_{\phi}(\pi_E^*f)$ vanish identically for all $\phi\in C_c(\R^n)$ if and only if the valuations
	\begin{align*}
		\mu_{\phi_1,\phi_2}(\pi_E^*f):=\int_{\R^n}\phi_1(\pi_E(x))\cdot \phi_2(\pi_{E^\perp}(x)) d\Psi(\pi_E^*f)=\KL_\Psi(E)\int_{\R^n}\phi_1 d\MA_E(f)\cdot \int_{E^\perp}\phi_2d\vol_{E^\perp}
	\end{align*}
	vanish identically for all $\phi_1\in C_c(E)$, $\phi_2\in C_c(E^\perp)$. If $\phi_1\ne 0$, then the valuation $\int_{\R^n}\phi_1 d\MA_E$ does not vanish identically, and we can choose $\phi_2\in C_c(E^\perp)$ with $\int_{E^\perp}\phi_2d\vol_{E^\perp}\ne 0$. Thus $\mu_{\phi_1,\phi_2}(\pi_E^*f)=0$ for all $f\in\Conv(E,\R)$ if and only if $\KL_\Psi(E)=0$. As this holds for all $E\in\Gr_k(\R^n)$, the claim follows.
\end{proof}
\begin{remark}
	Klain \cite{KlainshortproofHadwigers1995} has shown a similar characterization result for continuous, even, translation invariant valuations on convex bodies that are homogeneous of degree $k$. These similarities are not accidental and we will address them in a future article.
\end{remark}
If $G\subset \mathrm{O}(n)$ is a subgroup, then we call $\Psi\in \MAVal(\R^n)$ $G$-equivariant if
\begin{align*}
	g\cdot \Psi=\Psi\quad\text{for all}~g\in GF,
\end{align*}
where we consider the operation defined in \eqref{eq:defGlnROperation}. As the real Monge-Amp\`ere operator is $\mathrm{O}(n)$-equivariant, we directly obtain the following result.
\begin{corollary}
	\label{corollary:EquivKlainFunction}
	The map $\KL:\MAVal_k(\R^n)\rightarrow C(\Gr_k(\R^n))$ is $\mathrm{O}(n)$-equivariant and injective. In particular, $\Psi\in \MAVal_k(\R^n)$ is $G$-equivariant with respect to a subgroup $G\subset\mathrm{O}(n,\R)$ if and only if $\KL_\Psi$ is $G$-invariant.
\end{corollary}

\subsection{$\SO(n)$- and $\GL(n,\R)$-equivariant valuations}
\label{section:HadwigerTypeResult}
 Let us restate Theorem~\ref{maintheorem:HadwigerSOn} in the following way.
\begin{theorem}
	\label{theorem:HadwigerSOn}
	The space of $\SO(n)$-equivariant valuations in $\MAVal_k(\R^n)$ is $1$-dimensional and spanned by the $k$th Hessian measure.
\end{theorem}
\begin{proof}
	Let $\Psi\in\MAVal_k(\R^n)$ be $\SO(n)$-equivariant. By Corollary~\ref{corollary:EquivKlainFunction}, $\KL_\Psi\in C(\Gr_k(\R^n))$ is $\SO(n)$-invariant. As $\SO(n)$ operates transitively on $\Gr_k(\R^n)$, $\KL_\Psi$ is constant. However, $\Psi$ is uniquely determined by $\KL_\Psi$ by Proposition~\ref{proposition:KlainEmbedding}, so this space is at most $1$-dimensional. As the $k$th Hessian measure defines a non-trivial element in this space, we obtain the desired result.
\end{proof}
Let us say that $\Psi\in\MAVal_k(\R^n)$ is $\GL(n,\R)$-equivariant with weight $q\in\R$ if 
\begin{align*}
	g\cdot \Psi=|\det\nolimits_\R(g)|^q\Psi\quad\text{for all }g\in\GL(n,\R).
\end{align*}
\begin{theorem}
	\label{theorem:ClassificationGLnREquiv}
	$\Psi\in\MAVal_k(\R^n)$ is $\GL(n,\R)$-equivariant with weight $q\in \R$ if and only if
	\begin{itemize}
		\item $k=0$, $q=-1$ and $\Psi$ is a multiple of the Lebesgue measure, or
		\item $k=n$, $q=1$ and $\Psi$ is a multiple of the real Monge-Amp\`ere operator.
	\end{itemize}
\end{theorem}
\begin{proof}
	Obviously, the Lebesgue measure and the real Monge-Amp\`ere operator are $\GL(n,\R)$-equivariant with the given weights.\\
	
	Note that any $\GL(n,\R)$-equivariant $\Psi\in\MAVal_k(\R^n)$ is $\SO(n)$-equivariant and thus a multiple of the Hessian measure $\Phi_k$ by Theorem~\ref{theorem:HadwigerSOn}. We will thus show that none of these measures is $\GL(n,\R)$-equivariant with respect to any weight for $1\le k\le n-1$.\\
	
	Note that $(t Id)\cdot\Phi_k=t^{-(n-2k)}\Phi_k$ for $t>0$. In particular, $\Phi_k$ is not $\GL(n,\R)$-equivariant with weight $q=-1$. Split $\R^n=\R^{k}\times\R^{n-k}$ and let $g_t$ denote the multiplication with $t>0$ on the second factor. If $f\in \Conv(\R^{k},\R)$ and $\pi:\R^{k}\times \R^{n-k}\rightarrow\R^{k}$ is the projection onto the first factor, then $[\pi^*f]\circ g_t=\pi^*f$, and Proposition~\ref{proposition:KlainEmbedding} thus implies for Borel sets $A\subset\R^{k},B\subset \R^{n-k}$
	\begin{align*}
		[g_t\cdot \Phi_{k}](\pi^*f)[A\times B]=&\Phi_{k}([\pi^*f]\circ g_t)[g_t^{-1}(A\times B)]\\
		=&\Phi_{k}(\pi^*f_t)[g_t^{-1}(A\times B)]\\
		=&(\KL_{\Phi_{k}}\MA_{\R^{k}}(f)\otimes\vol_{n-k})[g^{-1}_t(A\times B)]\\
		=&t^{-(n-k)}(\KL_{\Phi_{k}}\MA_{\R^{k}}(f)\otimes\vol_{n-k})[A\times B]\\
		=&t^{-(n-k)}\Phi_{k}(\pi^*f)[A\times B].
	\end{align*}
	As $\det(g_t)=t^{n-k}$, $\Phi_k$ cannot be $\GL(n,\R)$-equivariant with respect to any weight $q\ne -1$. Thus, $\Phi_k$ is not $\GL(n,\R)$-equivariant with respect to any weight, which finishes the proof.
\end{proof}

\section{Characterization of $\MAVal_k(\R^n)$}
\label{section:characterization}
\subsection{Polynomials derived from primitive differential forms}
In this section we will discuss certain polynomials related to constant differential forms and $(k\times k)$-minors of symmetric matrices.\\

\begin{definition}
	Let $P\Lambda^{n-k,k}\subset\Lambda^{n-k}\R^n\otimes\Lambda^{k}(\R^n)^*$ denote the subspace of  complex-valued \emph{primitive} differential forms, that is, all $\tau\in \Lambda^{n-k}\R^n\otimes\Lambda^{k}(\R^n)^*$ such that $\omega_s\wedge\tau=0$, where $\omega_s$ denotes the natural symplectic form on $\R^n\times(\R^n)^*$.
\end{definition}
Due to the Lefschetz decomposition (see, for example, \cite[Proposition~1.2.30]{HuybrechtsComplexgeometry2005}), any differential form on $T^*\R^n$ can be uniquely written as a sum of a primitive differential form and a multiple of the symplectic form. For our purposes, it is thus sufficient to only consider primitive differential forms due to the fact that the differential cycle vanishes on multiples of the symplectic form, compare Theorem~\ref{theorem:characterization_differential_cycle}.\\

We extend elements of $\Lambda^{n-k,k}(\R^n\times(\R^n)^*)$ by $\C$-linearity to elements of $\Lambda^{n-k,k}(\C^n\times(\C^n)^*)$. In other words, we identify $\Lambda^{n-k,k}(\R^n\times(\R^n)^*)$ with the space of constant holomorphic forms on $\C^n\times(\C^n)^*$. $\GL(n,\C)$ operates on $\Lambda^{n-k,k}(\C^n\times(\C^n)^*)$ via the diagonal action on $\C^n\times (\C^n)^*$, that is,
\begin{align*}
	g\cdot \tau:=((g^{-1})^\#)^*\tau\quad \text{for}~\tau\in\Lambda^*(\C^n\times(\C^n)^*),~g\in\GL(n,\C),
\end{align*}
where $g^\#:\C^n\times(\C^n)^*\rightarrow\C^n\times(\C^n)^*$ is given by $g^\#(z,w)=(g^{-1}z,g^*w)$. Note that $\omega_s$ corresponds to a $\GL(n,\C)$-invariant holomorphic form. This implies the following.
\begin{lemma}
	$P\Lambda^{n-k,k}\subset\Lambda^n(\C^n\times(\C^n)^*)$ is a $\GL(n,\C)$-invariant subspace and thus a representation of $\GL(n,\C)$.
\end{lemma}

We will associate two different types of polynomials to elements of $P\Lambda^{n-k,k}$ related to $(k\times k)$-minors of different types of complex matrices. Let us first introduce some notation. For a complex vector space $V$, we denote by $\P^k(V)$ the space of $k$-homogeneous polynomials on $V$. Let us denote the space of symmetric $(n\times n)$-matrices by $SM_n$. If we let $\GL(n,\C)$ operate on $SM_n$ by $(g,A)\mapsto (g^T)^{-1}Ag^{-1}$ for $g\in\GL(n,\C)$, $A\in SM_n$, then $SM_n$ is isomorphic as a $\GL(n,\C)$-module to the space of symmetric bilinear forms on $\C^n$. More precisely, we identify $A\in SM_n$ with the bilinear form $(z,w)\mapsto z^TAw$ on $\C^n$. Let $\mathcal{M}_k\subset \P^k(SM_n)$ denote the space spanned by the $(k\times k)$-minors. This is a $\GL(n,\C)$-invariant subspace. We need the following result.
\begin{lemma}
	\label{lemma:kMinorsIrreducible}
	$\mathcal{M}_k$ is an irreducible representation of $\GL(n,\C)$.
\end{lemma} 
\begin{proof}
	It is well known that the space of all polynomials on $SM_n$ is a multiplicity free representation of $\GL(n,\C)$, see for example \cite[Theorem~5.7.1]{GoodmanWallachSymmetryrepresentationsinvariants2009}. Note that $\mathcal{M}_k$ is a rational representation of $\GL(n,\C)$ and thus decomposes into a direct sum of irreducible representations. Such a representation is uniquely determined by its highest weight vector and it is easy to see that the $k$th principal minor is a highest weight vector, see for example the proof of \cite[Theorem~5.7.3]{GoodmanWallachSymmetryrepresentationsinvariants2009}. Using permutation matrices, we see that the orbit of the $k$th principal minor spans $\mathcal{M}_k$. Thus $\mathcal{M}_k$ is an irreducible representation of $\GL(n,\C)$.
\end{proof}
The following result is well known, but we include a proof for the convenience of the reader.
\begin{lemma}
	\label{lemma:VanishingIdealSmallRank}
	The ideal $I_k\subset\mathcal{P}(SM_n)$ of all polynomials vanishing on matrices of rank strictly smaller than $k$ is prime and generated by $\mathcal{M}_k$.
\end{lemma}
\begin{proof}
	Note that $I_k$ is a $\GL(n,\C)$-invariant subspace of $\mathcal{P}(SM_n)$. As $\mathcal{P}(SM_n)$ is a rational  representation, $I_k$ decomposes into a direct sum of irreducible subspaces, which are each generated by the orbit of their highest weight vector. These vectors are given by products of principal minors, compare the proof of \cite[Theorem~5.7.3]{GoodmanWallachSymmetryrepresentationsinvariants2009}. Consequently, the highest weight vectors corresponding to the irreducible components of $I_k$ are divisible by a principle minor of size $l\times l$ where $l\ge k$. As any $(l\times l)$-minor, $l>k$, is contained in the ideal generated by $(k\times k)$-minors, all highest weight vectors in $I_k$ belong to the ideal generated by $\mathcal{M}_k$, and consequently, the same holds for the space spanned by their orbits, which is $I_k$.\\
	
	To see that $I_k$ is prime, assume that $p,q\in\mathcal{P}(SM_n)\setminus I_k$ satisfy $pq\in I_k$. Then there exist $A,B\in SM_n$ of rank at most $k-1$ such that $p(A)\ne 0$ and $q(B)\ne 0$. We can write \begin{align*}
		&A=\sum_{i=1}^{k-1}\lambda_i a_i\cdot a_i^T, &&B=\sum_{i=1}^{k-1}\eta_i b_i\cdot b_i^T
	\end{align*}
	for $\lambda_i,\eta_i\in\C$, $a_i,b_i\in \C^n$, $1\le i\le k-1$. For $t\in\C$ set \begin{align*}
		C(t):=\sum_{i=1}^{k-1}(t\lambda_i+(1-t)\eta_i) (ta_i+(1-t)b_i)\cdot (ta_i+(1-t)b_i).
	\end{align*} Then $p(C(t))$ and $q(C(t))$ are polynomials in $t$ that do not vanish identically. Consequently, there is $t_0\in \C$ such that $p(C(t_0))q(C(t_0))\ne 0$, which is a contradiction as the rank of $C(t)$ is at most $k-1$. Thus either $p\in I_k$ or $q\in I_k$.
\end{proof}

For $Q\in SM_n$ we let $Q(z):=\frac{1}{2}z^TQz$ denote the quadratic polynomial on $\C^n$ obtained from the corresponding quadratic form. In particular, we obtain a map 
\begin{align*}
	F_Q:\C^n\rightarrow& \C^n\times(\C^n)^*\\
	z\mapsto& (z,\partial Q(z)),	
\end{align*} where $\partial Q(z)$ denotes the holomorphic differential of $Q(z)$. In other words, $F_Q(z)=(z,z^TQ)$.
Let us denote the coordinates of $\C^n\times(\C^n)^*$ by $(z,w)$. As $\partial Q(z)$ is complex linear in $z$, $F_Q^*dw_i$ is again a constant differential form. This implies the following result.
\begin{lemma}
	For every $\tau\in P\Lambda^{n-k,k}$ there exists a unique polynomial $P_\tau\in \P^k(SM_n)$ such that
	\begin{align*}
		F_Q^*\tau=P_\tau(Q)dz_1\wedge\dots\wedge dz_n.
	\end{align*}
\end{lemma}
Note that $P_\tau$ vanishes on symmetric matrices of rank less than $k$: Any such matrix may be written as $Q=\sum_{j=1}^{k-1}\lambda_j e_j\cdot e_j^T$ for $e_1,\dots,e_{k-1}\in \C^n$ and $\lambda_1,\dots,\lambda_{k-1}\in \C$. Then it is easy to see that $F_Q^*(dw_{i_1}\wedge\dots\wedge dw_{i_k})=0$ for all $1\le i_1,\dots,i_k\le n$, which implies $F_Q^*\tau=0$.\\

If $Q\in SM_n$ is a real matrix, then $F_Q^*\tau$ may be interpreted as the restriction of the differential form $\tau$ to the graph of the differential $dQ(x)$ of the polynomial $Q(x)=\frac{1}{2}x^TQx$, $x\in\R^n$. Note that the tangent spaces of this graph are all isotropic, that is, the restriction of $\omega_s$ to these subspaces vanishes. Let us show that $\tau$ is uniquely determined by the polynomial $P_\tau$. We will need the following result.
\begin{lemma}[Bernig-Bröcker \cite{BernigBroeckerValuationsmanifoldsRumin2007} Lemma~1.4]
	\label{lemma:primitiveFormVanishing}
	Let $\beta\in \Lambda^k(\R^n\times(\R^n)^*)$ be primitive, $k\le n$. If $\beta$ vanishes on all isotropic $k$-dimensional linear subspaces, then $\beta=0$.
\end{lemma}
\begin{lemma}
	\label{lemma:injectivityPolySymKSym2}
	The map
	\begin{align*}
		P:P\Lambda^{n-k,k}&\rightarrow\P^k(SM_n)\\
		\tau&\mapsto P_\tau
	\end{align*}
	is injective and $\GL(n,\C)$-equivariant in the following sense: for $g\in\GL(n,\C)$ and $\tau\in P\Lambda^{n-k,k}$,
	\begin{align*}
		P_{g\cdot \tau}=\det(g)^{-1} \left(g\cdot P_\tau\right).
	\end{align*}
	In particular, the image of $P$ is a $\GL(n,\C)$-invariant subspace of $\P^k(SM_n)$.
\end{lemma}
\begin{proof}
	 Let us assume that $P_\tau=0$. By Lemma~\ref{lemma:primitiveFormVanishing}, it is sufficient to show that $\tau$ vanishes on all $n$-dimensional isotropic subspaces of $\R^n\times(\R^n)^*$.\\
	
	Consider the projection $\pi:\R^n\times(\R^n)^*\rightarrow\R^n$ onto the first factor. Let us call a Lagrangian subspace $E\subset\R^n\times (\R^n)^*$ \emph{regular} if $\pi|_E$ is injective. It is easy to see that $\tau$ vanishes on all $n$-dimensional Lagrangian subspaces if it vanishes on all regular $n$-dimensional Lagrangian subspaces. Let $E$ be such a regular space. As $\pi|_E$ is injective, we obtain a unique linear map $T:\R^n\rightarrow (\R^n)^*$ such that $(x,T(x))\in E$ for all $x\in\R^n$. As $E$ is Lagrangian,
	\begin{align*}
		0=\omega_s((x_1,T(x_1)),(x_2,T(x_2)))=\langle x_1, T(x_2)\rangle -\langle x_2, T(x_1)\rangle
	\end{align*}
	for all $x_1,x_2\in\R^n$, so $T:\R^n\rightarrow (\R^n)^*\cong \R^n$ is symmetric. We may thus choose a basis $e_1,\dots,e_n$ of $\R^n$ such that $T(e_i)=\lambda_ie_i$ for some $\lambda_i\in\R$ for all $1\le i\le n$. In other words, $E$ is spanned by the vectors $(e_1,\lambda_1e_1),\dots,(e_n,\lambda_ne_n)$. With respect to the basis $e_1,\dots, e_n$ of $\R^n$, define  $Q\in SM_n$ by $Q(x):=\sum_{i=1}^{n}\lambda_i \frac{x_i^2}{2}$. Then the tangent space of the graph of $x\mapsto dQ(x)$ in $(x,dQ(x))$ is $E$, and thus, the restriction of $\tau$ to $E$ vanishes as
	\begin{align*}
		F_Q^*\tau=P_\tau(Q)dx_1\wedge\dots\wedge dx_n=0.
	\end{align*}
	As this holds for all regular Langrangian subspaces $E$ of $\R^n\times(\R^n)^*$ and $\tau$ is primitive, Lemma~\ref{lemma:primitiveFormVanishing} implies that $\tau=0$.\\
	
	Let us show that $P$ is $\GL(n,\C)$-equivariant in the sense stated above. For $g\in\GL(n,\C)$ let $G_g:\C^n\times(\C^n)^*\rightarrow\C^n\times(\C^n)^*$ denote the diagonal operation $(z,w)\mapsto (gz,(g^{-1})^*w)$. Then
	\begin{align*}
		G_g\circ F_Q(z)=&(g z, \partial Q(z)\circ g^{-1})=(gz, z^TQg^{-1})=(gz, (gz)^T(g^{-1})^TQg^{-1})=(gz,\partial(g\cdot Q)(gz))\\
		=&F_{g\cdot Q}(gz).
	\end{align*}
	If $\tilde{G}_g:\C^n\rightarrow\C^n$ denotes the multiplication with $g\in\GL(n,\C)$, we thus have $G_g\circ F_Q=F_{g\cdot Q}\circ \tilde{G}_g$. For $\tau\in P\Lambda^{n-k,k}$ this implies
	\begin{align*}
		P_{G_g^*\tau}dz_1\wedge\dots \wedge dz_n=&F_Q^*(G_g^*\tau)=\tilde{G}_g^*\left(F_{g\cdot Q}^*\right)=P_\tau(g\cdot Q)\tilde{G}_g^*(dz_1\wedge\dots\wedge dz_n)\\
		=&P_\tau(g\cdot Q)\det(g)dz_1\wedge\dots\wedge dz_n.
	\end{align*}
	As $\tilde{G}_g=(g^{-1})^{\#}$,
	\begin{align*}
		P_{g\cdot \tau}(Q)=P_\tau(g^{-1}\cdot Q)\det(g)^{-1}=\det(g)^{-1} \left(g\cdot P_\tau\right)(Q).
	\end{align*}
\end{proof}
\begin{proposition}
	\label{proposition:ImageP}
	The image of $P$ coincides with $\mathcal{M}_k$.
\end{proposition}
\begin{proof}
	Let us first show that the image of $P$ is contained in the space spanned by $(k\times k)$-minors. First, consider the ideal $I_k$ of all polynomials on $SM_n$ that vanish on elements of $SM_n$ of rank at most $k-1$. This is a prime ideal generated by the $(k\times k)$-minors, compare Lemma~\ref{lemma:VanishingIdealSmallRank}. Obviously $P_\tau\in I_k$ for $\tau\in P\Lambda^{n-k,k}$. As $P_\tau$ is a polynomial of degree $k$, we can thus express $P_\tau$ as a $\C$-linear combination of the $(k\times k)$-minors.\\
	
	Next, note that $\mathcal{M}_k$ is an irreducible representation of $\GL(n,\C)$ by Lemma~\ref{lemma:kMinorsIrreducible}. As the image of $P$ is a non-trivial $\GL(n,\C)$-submodule of $\mathcal{M}_k$ by Lemma~\ref{lemma:injectivityPolySymKSym2}, these two spaces coincide.
\end{proof}
Combining this result with Lemma~\ref{lemma:kMinorsIrreducible}, we obtain the following.
\begin{corollary}
	\label{corollary:primitiveFormsIrreducible}
	$P\Lambda^{n-k,k}$ is an irreducible representation of $\GL(n,\C)$.
\end{corollary}

Let us now turn to a second family of polynomials obtained from elements of $P\Lambda^{n-k,k}$. Let us denote the standard coordinates on $T^*\R^n=\R^n\times (\R^n)^*$ by $(x,y)$.
\begin{definition}
	If $\tau\in P\Lambda^{n-k,k}$, then we define a complex polynomial $Q_\tau$ on $(\C^n)^k$ by
	\begin{align}
		\label{eq:defQtau}
		Q_\tau(w_1,\dots,w_k)dx_1\wedge\dots\wedge dx_n= \mathcal{L}_{X_1}\dots\mathcal{L}_{X_k}\tau,
	\end{align}
	where $X_j$ denotes the symplectic vector field of $x\mapsto -\frac{1}{2}\langle w_j,x\rangle^2$, that is, the unique vector field on $\R^n\times(\R^n)^*$  with complex coefficients such that $i_{X_j}\omega_s=-\frac{1}{2}d\langle w_j,x\rangle^2$.
\end{definition}
Here $\mathcal{L}_X=d\circ i_X+i_X\circ d$ denotes the complex-linear extension of the Lie-derivative. Note that $X_j$ is given by
\begin{align*}
	X_j=\langle w_j,x\rangle\sum_{l=1}^{n} w_{j,l}\frac{\partial}{\partial y_l},
\end{align*}
so the right hand side of \eqref{eq:defQtau} defines a polynomial. It is not difficult to see that this definition coincides with the following characterization:

For $w_1,\dots,w_k\in\C^n$ let $\tilde{X}_j$ denote the unique vector field on $\R^n\times(\R^n)^*$ with complex coefficients such that
\begin{align*}
	i_{\tilde{X}_j}\omega_s&=d\left(\exp\left(i\langle w_j,x\rangle\right)\right).
\end{align*}
If we write $w_j=(w_{j,1},\dots, w_{j,n})\in\C^n$, then
\begin{align*}
	\tilde{X}_j&=-i\exp\left(i\langle w_j,x\rangle\right)\sum_{l=1}^{n}w_{j,l}\frac{\partial}{\partial y_l},
\end{align*}
and $Q_\tau$ is characterized by
\begin{align}
	\label{eq:QtauFourier-Laplace}
	\exp\left(i\left\langle \sum_{j=1}^{k}w_j,x\right\rangle\right)Q_\tau(w_1,\dots,w_k)dx_1\wedge\dots\wedge dx_n=\mathcal{L}_{\tilde{X}_1}\dots\mathcal{L}_{\tilde{X}_k}\tau.
\end{align}
\begin{lemma}
	$Q_\tau$ is symmetric for any $\tau\in P\Lambda^{n-k,k}$.
\end{lemma}
\begin{proof}
	The vector fields $X_j$ commute.
\end{proof}
\begin{proposition}
	\label{prop:relationPQ}
	Let 
	\begin{align*}
		\Phi:(\C^n)^k&\rightarrow SM_n\\
		(w_1,\dots,w_k)&\mapsto \sum_{i=1}^{k}w_i\cdot w_i^T.
	\end{align*}
	Then for all $\tau\in P\Lambda^{n-k,k}$
	\begin{align*}
		\Phi^*P_\tau=\frac{1}{k!}Q_\tau.
	\end{align*}
\end{proposition}	
\begin{proof}
	Fix $\tau\in P\Lambda^{n-k,k}$. For $(w_1,\dots,w_k)\in(\C^n)^k$ set $Q_i(z):=\frac{1}{2}\langle w_i,z\rangle^2$. For $\lambda_i\in\C$ we define $G^{\lambda_i}_{Q_i}: \C^n\times(\C^n)^*\rightarrow\C^n\times(\C^n)^*$, $(z,w)\mapsto (z,w+\lambda_i\partial Q_i(z))$. Then \begin{align*}
		F_{\sum_{i}\lambda_iQ_i}=G^{\lambda_k}_{Q_k}\circ\dots \circ G^{\lambda_1}_{Q_1}\circ F_0.
	\end{align*} Next, it follows from the definition that $(\lambda_1,\dots,\lambda_k)\mapsto (F_{\sum_{i}\lambda_iQ_i})^*\tau$ is a homogeneous polynomial of degree $k$. If we consider a coefficient of this polynomial that does not correspond to a monomial involving $\lambda_j$, then this coefficient occurs in the polynomial 
	\begin{align*}
		(F_{\sum_{i\ne j}\lambda_iQ_i})^*\tau.
	\end{align*}
	But the rank of $\sum_{i\ne j}\lambda_iQ_i$ is at most $k-1$, so this expression vanishes identically. In other words, the only non-trivial coefficient of this polynomial is the coefficient belonging to $\lambda_1\dots\lambda_k$. Thus
	\begin{align*}
		(F_{\Phi(w_1,\dots,w_k)})^*\tau=&(F_{\sum_{i}\lambda_iQ_i})^*\tau|_{\lambda_1,\dots,\lambda_k=1}=\left(\lambda_1\dots\lambda_k\frac{1}{k!}\frac{\lambda^k}{\partial\lambda_1\dots\partial\lambda_k}\Big|_0(F_{\sum_{i}\lambda_iQ_i})^*\tau\right)\Big|_{\lambda_1,\dots,\lambda_k=1}\\
		=&\frac{1}{k!}\frac{\lambda^k}{\partial\lambda_1\dots\partial\lambda_k}\Big|_0 F_0^*\left((G^{\lambda_1}_{Q_1})^*\dots (G^{\lambda_k}_{Q_k})^*\tau\right)\\
		=&\frac{1}{k!} F_0^*\left(\mathcal{L}_{X'_1}\dots\mathcal{L}_{X'_k}\tau\right),
	\end{align*}
	where $X'_i:=\frac{d}{d\lambda}|_0G^\lambda_{Q_i}$. Let $\alpha=\sum_{i=1}^{n}w_idz_i$ denote the holomorphic extension of the natural $1$-form on $T^*\R^n= \R^n\times(\R^n)^*$ to $\C^n\times(\C^n)^*$. Note that $(G_{Q_i}^\lambda)^*\alpha=\alpha+\lambda \partial Q_i(z)$ by definition. As $\pi\circ G_{Q_i}^\lambda=Id_{\C^n}$, we also have $i_{X'_i}\alpha|_{(z,w)}=\langle w,d\pi(X'_i)\rangle=0$, so
	\begin{align*}
		\partial Q_j(z) =\mathcal{L}_{X'_j}\alpha=i_{X'_j}d\alpha=-i_{X'_j}\omega_s.
	\end{align*}
	In particular, $X'_j=X_j$. We thus obtain
	\begin{align*}
		P_\tau(\Phi(w_1,\dots,w_k))dz_1\wedge\dots\wedge dz_n=&\left(F_{\Phi(w_1,\dots,w_k)}\right)^*\tau
		=\frac{1}{k!} F_0^*\left(\mathcal{L}_{X_1}\dots\mathcal{L}_{X_k}\tau\right)\\
		=&\frac{1}{k!}Q_\tau(w_1,\dots,w_k)F_0^*(dz_1\wedge\dots\wedge dz_n)\\
		=&\frac{1}{k!}Q_\tau(w_1,\dots,w_k)dz_1\wedge\dots\wedge dz_n.
	\end{align*}
\end{proof}
Let $\GL(n,\C)$ operate on $(\C^n)^k$ by $g\cdot (w_1,\dots,w_k):=((g^{-1})^Tw_1,\dots,(g^{-1})^Tw_k)$, that is, we consider $\C^n\cong(\C^n)^*$ using the standard bilinear pairing on $\C^n $. From the definition of $\Phi$, we directly obtain the following.
\begin{lemma}
	$\Phi:(\C^n)^k\rightarrow SM_n$ is $\GL(n,\C)$-equivariant.
\end{lemma}
\begin{corollary}
	$Q_\tau=0$ if and only if $\tau=0$. 
\end{corollary}
\begin{proof}
	Due to Proposition~\ref{prop:relationPQ}, it is sufficient to show that $\Phi^*P_\tau=0$ implies $P_\tau=0$. As $P_\tau\in\mathcal{M}_k$  by Proposition~\ref{proposition:ImageP}, which is an irreducible subspace by Lemma~\ref{lemma:kMinorsIrreducible}, the claim follows from the fact that $\Phi$, and therefore $\Phi^*$, is $\GL(n,\C)$ equivariant, so the restriction of $\Phi^*$ to $\mathcal{M}_k$ is either injective or identically zero. Obviously it is not identically zero, so the restriction of $\Phi^*$ to this space is injective.
\end{proof}
Let $\mathcal{M}^2_k\subset\P((\C^n)^k)$ denote the subspace spanned by quadratic products of the $(k\times k)$-minors of a matrix in $(\C^n)^k$. It is easy to see that this is a $\GL(n,\C)$-invariant subspace of $\mathcal{P}^{2k}((\C^n)^k)$. We will call a polynomial in $\mathcal{P}((\C^n)^k)$ homogeneous of degree $(a_1,\dots,a_k)\in\mathbb{N}^k$ if it is homogeneous with respect to $a_i$ in its $i$th argument.
\begin{proposition}
	\label{prop:PolynomialsFormPrimitiveDiff}
	The image of $Q$ coincides with $\mathcal{M}^2_k$.
\end{proposition}
\begin{proof}
	Let $\Delta_\alpha$ denote the family of $(k\times k)$-minors of elements of $(\C^n)^k$. We will us first show that the image of $Q$ is contained in $\mathcal{M}_k^2$. Due to Proposition~\ref{prop:relationPQ}, we have to show that $\Phi^*p\in\mathcal{M}^2_k$ for every $p\in\mathcal{M}_k$. First, note that $\Phi^*p$ vanishes on the common zero sets of the $(k\times k)$-minors $\Delta_\alpha$: If $w_1,\dots,w_k\in\C^n$ are linearly dependent, then the matrix $\sum_{i=1}^{k}w_i\cdot w_i^T$ is of rank at most $k-1$, so $p$ vanishes on this matrix. Thus $\Phi^*p$ belongs to the ideal of all polynomials on $(\C^n)^k$ that vanish on matrices of rank at most $k-1$. This is a prime ideal generated by the $(k\times k)$-minors, compare \cite[Section~16]{MillerSturmfelsCombinatorialcommutativealgebra2005}. By construction, $\Phi^*p$ is a symmetric and homogeneous polynomial of degree $(2,\dots,2)$ on $(\C^n)^k$. If we express
	\begin{align*}
		\Phi^*p=\sum_{\alpha}h_\alpha\cdot \Delta_\alpha
	\end{align*}
	as a $\P((\C^n)^k)$-linear combination of the $(k\times k)$-minors $\Delta_\alpha$, we may thus assume that $h_\alpha\in\P^k((\C^n)^k)$ is $1$-homogeneous in each argument and thus multilinear. As $\Delta_\alpha$ is a skew-symmetric polynomial with respect to permutations of the arguments $z_1,\dots,z_k$, and $\Phi^*p$ is symmetric, we may further assume that $h_\alpha$ is skew-symmetric. But then $h_\alpha$ is a multilinear, skew-symmetric polynomial and thus a linear combination of the $(k\times k)$-minors $\Delta_\alpha$.\\
	
	Consequently, the image of $Q$ is a subspace of dimension $\dim P\Lambda^{n-k,k}=\binom{n}{k}^2-\binom{n}{k-1}\binom{n}{n-k-1}$. As the dimension of the space spanned by quadratic products of the $(k\times k)$-minors is at most  $\binom{n}{k}^2-\binom{n}{k-1}\binom{n}{n-k-1}$ due to the Grassmann-Plücker relations (compare for example \cite[Section~14]{MillerSturmfelsCombinatorialcommutativealgebra2005}), the claim follows.
\end{proof}

\subsection{Relation to measure-valued valuations obtained from the differential cycle}
For $\Psi\in\MAVal_k(\R^n)$, $\phi\in C_c(\R^n)$, define $\Psi[\phi]\in \VConv_k(\R^n)$ by 
\begin{align*}
	\Psi[\phi](f):=\int_{\R^n}\phi d\Psi(f).
\end{align*}
Our interest in these polynomials stems from the following observation:
\begin{proposition}
	\label{prop:FormulaQFourier-LaplaceDiffForm}
	If $\tau\in P\Lambda^{n-k,k}$, then $\Psi_\tau(f)[B]=D(f)[1_{\pi^{-1}(B)}\tau]$ satisfies
	\begin{align*}
		\mathcal{F}(\GW(\Psi_\tau[\phi]))[z_1,\dots,z_k]=\frac{(-1)^k}{k!}Q_\tau(z_1,\dots,z_k)\mathcal{F}(\phi)\left[\sum_{i=1}^{k}z_i\right]
	\end{align*}
	for all $z_1,\dots,z_k\in \C^n$.
\end{proposition}
\begin{proof}
	It is sufficient to show this equation for $z_1=ix_1\dots,z_k=ix_k$ for $x_1,\dots,x_k\in \R^n$. Let $\phi\in C_c(\R^n)$. Then $\Psi_\tau[\phi]\in\VConv_k(\R^n)$ is a continuous valuation and we calculate
	\begin{align*}
		\mathcal{F}(\GW(\Psi_\tau[\phi]))[ix_1,\dots,ix_k]=&\GW(\Psi_\tau[\phi])[\exp(-\langle x_1,\cdot\rangle)\otimes\dots\otimes \exp(-\langle x_1,\cdot\rangle)]\\
		=&\overline{\Psi_\tau[\phi]}(\exp(-\langle x_1,\cdot\rangle),\dots, \exp(-\langle x_1,\cdot\rangle))\\
		=&\frac{1}{k!}\frac{\partial^k}{\partial \lambda_1\dots\partial\lambda_k}\Big|_0\Psi_\tau[\phi]\left(\sum_{i=0}^{k}\lambda_i\exp(-\langle x_i,\cdot\rangle)\right),
	\end{align*}
	where we have used the defining property of the Goodey-Weil distribution in Theorem~\ref{theorem:GW} and the definition of the polarization $\overline{\Psi_\tau[\phi]}$. Set $h_i:=\exp(-\langle x_i,\cdot\rangle)$. As $\Psi_\tau[\phi](f)=D(f)[\pi^*\phi\wedge \tau]$, we may apply Proposition~\ref{proposition:differentialCycleSum} to obtain
	\begin{align*}
		\frac{1}{k!}\frac{\partial^k}{\partial \lambda_1\dots\partial\lambda_k}\Big|_0\Psi_\tau[\phi]\left(\sum_{i=0}^{k}\lambda_i\exp(-\langle x_i,\cdot\rangle)\right)
		=&\frac{1}{k!}\frac{\partial^k}{\partial \lambda_1\dots\partial\lambda_k}\Big|_0D\left(0+\sum_{i=1}^{k}\lambda_ih_i\right)[\pi^*\phi\wedge\tau]\\
		=&\frac{1}{k!}\frac{\partial^k}{\partial \lambda_1\dots\partial\lambda_k}\Big|_0D\left(0\right)[\pi^*\phi \wedge G_{\lambda_1h_1}^*\dots G_{\lambda_1h_1}^*\tau]\\
		=&\frac{1}{k!}D\left(0\right)[\pi^*\phi \wedge\mathcal{L}_{X_{h_1}}\dots \mathcal{L}_{X_{h_k}}\tau],
	\end{align*}
	where $X_{h_j}:=\frac{d}{dt}|_0G_{th_j}$. As in the proof of Proposition~\ref{prop:relationPQ}, $G_{th_j}^*\alpha=\alpha+tdh_j$, so we obtain
	\begin{align*}
		dh_j=\mathcal{L}_{X_{h_j}}\alpha=-i_{X_{h_j}}\omega_s.
	\end{align*}
	Note that
	\begin{align*}
		dh_j=d\exp(-\langle x_j,\cdot\rangle)=d\exp(i\langle ix_j,\cdot\rangle).
	\end{align*}
	The characterization of $Q_\tau$ in \eqref{eq:QtauFourier-Laplace} implies
	\begin{align*}
		\mathcal{L}_{X_{h_1}}\dots \mathcal{L}_{X_{h_k}}\tau=(-1)^k\exp\left(i\left\langle\sum_{j=1}^{k} ix_j,\cdot\right\rangle\right)Q_\tau(ix_1,\dots,ix_k)dx_1\wedge\dots\wedge dx_n.
	\end{align*}
	Thus,
	\begin{align*}
		&\mathcal{F}(\GW(\Psi_\tau[\phi]))[ix_1,\dots,x_k]\\
		=&\frac{(-1)^k}{k!}Q_\tau(ix_1,\dots,ix_k)D(0)\left[\pi^*\phi\cdot \pi^*\exp\left(i\left\langle\sum_{j=1}^{k} ix_j,\cdot\right\rangle\right)\wedge dx_1\wedge \dots \wedge dx_n\right]\\
		=&\frac{(-1)^k}{k!}Q_\tau(ix_1,\dots,ix_k)\mathcal{F}(\phi)\left[\sum_{j=1}^{k}ix_k\right],
	\end{align*}
	as $D(0)$ is given by integration over $\{(x,0)\in T^*\R^n: x\in \R^n\}$.
\end{proof}	
\begin{proposition}
	\label{proposition:C2FormulaCurvatureMeasurePolynomial}
	For $\tau\in P\Lambda^{n-k,k}$ the map $\Psi_\tau(f)[B]=D(f)[1_{\pi^{-1}(B)}\tau] $ is given for $f\in\Conv(\R^n,\R)\cap C^2(\R^n)$ by
	\begin{align*}
		\Psi_\tau(f)[B]=\int_{B}P_\tau(D^2f(x))d\vol_n(x)\quad\text{for all bounded Borel sets } B\subset\R^n.
	\end{align*}
\end{proposition}
\begin{proof}
	Consider the map
	\begin{align*}
		\tilde{G}_f:\R^n\rightarrow&\R^n\times (\R^n)^*\\
		x\mapsto& (x,df(x)).
	\end{align*}
	For $x_0\in \R^n$ define $Q_{x_0}\in \Sym^2(\R^n)$ by $Q_{x_0}(x):=\frac{1}{2}\langle D^2f(x_0)x,x\rangle$. Note that
	\begin{align*}
		\tilde{G}_f^*dy_i|_{x_0}=d\left(\frac{\partial f}{\partial x_i}\right)\Big|_{x_0}=\sum_{j=1}^{n}\frac{\partial^2f}{\partial x_i\partial x_j}(x_0)dx_j=G_{Q_{x_0}}^*dy_i.
	\end{align*}
	Thus
	\begin{align*}
		\tilde{G}_f^*\tau|_{x_0}=G_{Q_{x_0}}^*\tau|_{x_0}=P_\tau(Q_{x_0})dx_1\wedge\dots\wedge dx_n=P_\tau(D^2f(x_0))dx_1\wedge\dots\wedge dx_n,
	\end{align*}
	which implies
	\begin{align*}
		\Psi_\tau(f)[B]=D(f)[1_{\pi^{-1}(B)}\tau]=\int_{B}\tilde{G}_f^*\tau=\int_{B}P_\tau (D^2(f(x))d\vol_n(x).
	\end{align*}
\end{proof}

\subsection{Classification result}
	\begin{theorem}
		\label{theorem:FormulaFourier-LaplaceCurvatureMeasures}
		For every $\Psi\in\MAVal_k(\R^n)$ there exists a unique homogeneous, symmetric polynomial $Q[\Psi]$ on $(\C^n)^k$ of degree $(2,\dots,2)$ such that
		\begin{align*}
			\mathcal{F}(\GW(\Psi[\phi]))[z_1,\dots,z_k]=\frac{(-1)^k}{k!}Q[\Psi](z_1,\dots,z_k)\mathcal{F}(\phi)\left[\sum_{j=1}^{k}z_j\right].
		\end{align*}
		In particular, $Q[\Psi]=0$ if and only if $\Psi=0$.
	\end{theorem}
	\begin{proof}
		As $\mathcal{F}\circ\GW$ is injective, it is clear that the representation (and thus the polynomial $Q[\Psi]$) is unique if it exists.\\
		
		If $y_1,\dots,y_k\in E$ for $E\in\Gr_k(\R^n)$, then by Proposition~\ref{proposition:KlainEmbedding}
		\begin{align*}
			\Psi[\phi]\left(\sum_{j=1}^{k}\lambda_j\exp\langle y_j,x\rangle \right)=&\KL_\Psi(E)\int_{\R^n}\phi(x) d\left(\MA_E\left(\sum_{j=1}^{k}\lambda_j\exp\langle y_j,x\rangle \right)\otimes \vol_{E\perp}\right),
		\end{align*}
		so
		\begin{align*}
			&\GW(\Psi[\phi])[\exp(\langle y_1,\cdot\rangle)\otimes\dots\otimes\exp(\langle y_k,\cdot\rangle)]\\
			=&\KL_\Psi(E)\frac{1}{k!}\frac{\partial^k}{\partial\lambda_1\dots\partial\lambda_k}\Big|_0\int_{\R^n}\phi(x) d\left(\MA_E\left(\sum_{j=1}^{k}\lambda_j\exp\langle y_j,x\rangle \right)\otimes \vol_{E\perp}\right)\\
			=&\KL_\Psi(E)\frac{1}{k!}\det(\langle y_i,y_j\rangle)_{i,j=1}^k\mathcal{F}(\phi)\left[-i\sum_{j=1}^{k}y_j\right].
		\end{align*}
		Choose $\phi\in C_c(\R^n)$ with $\mathcal{F}(\phi)[0]\ne0$. Then $Q[\Psi]$ given by
		\begin{align*}
			Q[\Psi](z_1,\dots,z_k):=(-1)^kk!\frac{\mathcal{F}(\GW(\Psi[\phi]))[z_1,\dots,z_k]}{\mathcal{F}(\phi)\left[\sum_{j=1}^{k}z_k\right]}
		\end{align*}
		defines a holomorphic function in a neighborhood of $0\in(\C^n)^k$		
		 that satisfies for $y_1,\dots,y_k\in E\in\Gr_k(\R^n)$ with $|y_1|,\dots, |y_k|$ small enough
		\begin{align}
			\label{equation:PPsiOnRealSubspace}
			Q[\Psi](y_1,\dots,y_k)=\KL_\Psi(E)\det(\langle y_i,y_j\rangle)_{i,j=1}^k.
		\end{align}
		As a holomorphic function is uniquely determined by its restriction to the real subspace $i(\R^n)^k\subset(\C^n)^k$, we see that the germ of $Q[\Psi]$ does not depend on the choice of $\phi$ with $\mathcal{F}(\phi)[0]\ne 0$. By rescaling the argument of $\phi$ and using the equivariance of the Fourier-Laplace transform, we thus see that $Q[\Psi]$ defines an entire function on $(\C^n)^k$ such that \eqref{equation:PPsiOnRealSubspace} holds. In particular,
		\begin{align}
			\label{eq:DecompFourier-LaplacePolynomial}
			\mathcal{F}(\GW(\Psi[\phi]))[z_1,\dots,z_k]=\frac{(-1)^k}{k!}Q[\Psi](z_1,\dots,z_k)\mathcal{F}(\phi)\left[\sum_{j=1}^{k}z_j\right]
		\end{align}
		for all $\phi\in C_c(\R^n)$ with $\mathcal{F}(\phi)[0]\ne 0$. If $\phi\in C_c(\R^n)$  satisfies $\mathcal{F}(\phi)[0]=0$, we can replace $\phi$ by $\phi+\epsilon \phi_0$, where $\mathcal{F}(\phi_0)[0]\ne 0$, and let $\epsilon$ go to zero. Thus \eqref{eq:DecompFourier-LaplacePolynomial} holds for all $\phi\in C_c(\R^n)$. It remains to see that $Q[\Psi]$ is a homogeneous polynomial of degree $2k$.\\
		
		Note that for $y_1,\dots,y_k\in E\in\Gr_k(\R^n)$ and $t_1,\dots,t_k\in\R$
		\begin{align*}
			Q[\Psi](t_1iy_1,\dots,t_kiy_k)=\KL(E)\det(\langle t_iy_i,t_jy_j\rangle)_{i,j=1}^k=Q[\Psi](iy_1,\dots,iy_k)\prod_{j=1}^{k}t_j^2,
		\end{align*}
		so $Q[\Psi](t_1z_1,\dots,t_kz_k)=Q[\Psi](z_1,\dots,z_k)\prod_{j=1}^{k}t_j^2$ for all $z_1,\dots,z_k\in\C^n$ by the identity theorem. In particular,
		\begin{align*}
			|Q[\Psi](z_1,\dots,z_k)|\le \prod_{j=1}^{k}|z_j|^2 \sup_{|w_1|\le 1,\dots,|w_k|\le 1} |Q[\Psi](w_1,\dots,w_k)|.
		\end{align*}
		Thus $Q[\Psi]$ is an entire function that is bounded by a polynomial of degree $2k$ and thus a polynomial of degree at most $2k$ itself. From \eqref{equation:PPsiOnRealSubspace} we deduce that $Q[\Psi]$ is in fact a polynomial of degree $(2,\dots,2)$ and that $Q[\Psi]$ is symmetric.\\
		
		For the last claim, observe that $Q[\Psi]=0$ implies $\mathcal{F}(\GW(\Psi[\phi]))=0$ for all $\phi\in C_c(\R^n)$. As $\mathcal{F}\circ \GW$ is injective, $\Psi[\phi]=0$ for all $\phi\in C_c(\R^n)$, which implies $\Psi=0$.
	\end{proof}
	\begin{corollary}
		\label{corollary:PPsi_kMinors}
		For $\Psi\in\MAVal_k(\R^n)$ the polynomial $Q[\Psi]$ belongs to $\mathcal{M}^2_k$.
	\end{corollary}
	\begin{proof}
		Let $\Delta_\alpha$ denote a basis of the space spanned by $(k\times k)$-minors.
		We will first show that $Q[\Psi]$ is contained in the ideal generated by the $(k\times k)$-minors on $(\C^n)^k$. This ideal coincides with the ideal of all polynomials that vanish on all $(z_1,\dots,z_k)\in(\C^n)^k$ that are linearly dependent. As 
		\begin{align*}
			Q[\Psi](z_1,\dots,z_k)=(-1)^kk!\frac{\mathcal{F}(\GW(\Psi[\phi])[z_1,\dots,z_k]}{\mathcal{F}(\phi)\left[\sum_{i=1}^{k}z_i\right]}
		\end{align*} 
		for all $\phi\in C_c(\R^n)$ with $\mathcal{F}(\phi)\left[\sum_{i=1}^{k}z_i\right]\ne 0$ by Theorem~\ref{theorem:FormulaFourier-LaplaceCurvatureMeasures}, the claim follows because $\GW(\Psi[\phi])[z_1,\dots,z_k]$ vanishes on linearly dependent vectors by Lemma~\ref{lemma:ZeroSetFourier-Laplace}.\\
		
		We therefore find $h_\alpha\in\C[z_1,\dots,z_n]$ such that
		\begin{align*}
			Q[\Psi](z_1,\dots,z_n)=\sum_{\alpha}h_\alpha\cdot\Delta_\alpha.
		\end{align*}
		As $Q[\Psi]$ is a homogeneous polynomial of degree $(2,\dots,2)$ by Theorem~\ref{theorem_embedding_Cont}, we may assume that $h_\alpha$ is homogeneous of degree $(1,\dots,1)$. As $\Delta_\alpha$ is skew-symmetric with respect to permutations of the arguments and $Q[\Psi]$ is symmetric by Theorem~\ref{theorem:FormulaFourier-LaplaceCurvatureMeasures}, we can also assume that $h_\alpha$ is skew-symmetric. But then $h_\alpha$ is an alternating multilinear functional on $(\C^n)^k$ and thus a linear combination of $(k\times k)$-minors, which shows the claim.
	\end{proof} 
	\begin{theorem}
		\label{theorem:RepresentationDifferentialForms}
		For every $\Psi\in \MAVal_k(\R^n)$ there exists a unique primitive differential form $\tau\in P\Lambda^{n-k,k}$ such that $\Psi=\Psi_\tau$.
	\end{theorem}
	\begin{proof}
		By Theorem~\ref{theorem:FormulaFourier-LaplaceCurvatureMeasures}, $\Psi\in\MAVal_k(\R^n)$ is uniquely determined by the polynomial $Q[\Psi]$. This polynomial is contained in the space spanned by squares of $(k\times k)$-minors due to Corollary~\ref{corollary:PPsi_kMinors}. Proposition~\ref{prop:PolynomialsFormPrimitiveDiff} thus implies that there exists a differential form $\tau\in P\Lambda^{n-k,k}$ with $Q_\tau=Q[\Psi]$. If we consider $\Psi_\tau\in\MAVal_k(\R^n)$, then Proposition~\ref{prop:FormulaQFourier-LaplaceDiffForm} and the definition of $Q[\Psi]$ in Theorem~\ref{theorem:FormulaFourier-LaplaceCurvatureMeasures} show that 
		\begin{align*}
			\mathcal{F}(\GW(\Psi[\phi]))=\mathcal{F}(\GW(\Psi_\tau[\phi])
		\end{align*}
		for all $\phi\in C_c(\R^n)$. Thus $\Psi=\Psi_\tau$.
	\end{proof}
		Let us note that the previous argument boils down to the following statement.
	\begin{corollary}
		\label{corollary:CharacterizationInTermsOfCharacteristicPolynomial}
		The map 
		\begin{align*}
			\MAVal_k(\R^n)&\rightarrow \mathcal{M}^2_k\\
			\Psi&\mapsto Q[\Psi] 
		\end{align*}
		is bijective.
	\end{corollary}
	\begin{proof}
		The map well defined by Corollary~\ref{corollary:PPsi_kMinors}, onto by Proposition~\ref{prop:PolynomialsFormPrimitiveDiff} and Proposition~\ref{prop:FormulaQFourier-LaplaceDiffForm}, and injective by Theorem~\ref{theorem:FormulaFourier-LaplaceCurvatureMeasures}.
	\end{proof}
	\begin{proof}[Proof of Theorem~\ref{maintheorem:IrreducibilityMAVal}]
		Note that the map
		\begin{align*}
			P\Lambda^{n-k,k}&\rightarrow\MAVal_k(\R^n)\\
			\tau&\mapsto \Psi_\tau
		\end{align*}
		is $\GL(n,\R)$-equivariant (up to sign) due to Corollary~\ref{corollary:equivarianceMeasuresFromDifferentialForms} and bijective by Theorem~\ref{theorem:RepresentationDifferentialForms}. As $P\Lambda^{n-k,k}$ is a finite dimensional representation of $\GL(n,\R)$, it is thus sufficient to show that $P\Lambda^{n-k,k}$ is an irreducible representation of $\GL(n,\R)$. Recall that we extended elements of $P\Lambda^{n-k,k}$ to constant holomorphic forms of $\C^n\times(\C^n)^*$ by $\C$-linearity and that $P\Lambda^{n-k,k}$ is an irreducible rational representation of $\GL(n,\C)$ under this identification by Corollary~\ref{corollary:primitiveFormsIrreducible}. It is well known that this implies that $P\Lambda^{n-k,k}$ is an irreducible representation of $\GL(n,\R)$. Indeed, let $W\subset P\Lambda^{n-k,k}$ be a non-trivial $\GL(n,\R)$-invariant subspace. As $P\Lambda^{n-k,k}$ is a rational representation of $\GL(n,\C)$, the set
		\begin{align*}
			\{g\in\GL(n,\C): g\cdot\tau\in W~\text{for all}~\tau\in W\}\subset \GL(n,\C)
		\end{align*}
		is Zariski closed. By assumption, it contains the Zariski dense subset $\GL(n,\R)$ and thus coincides with $\GL(n,\C)$. In other words, $W$ is a non-trivial $\GL(n,\C)$-invariant subspace of $P\Lambda^{n-k,k}$ and therefore these spaces coincide. Thus $P\Lambda^{n-k,k}$ is an irreducible representation of $\GL(n,\R)$. 
	\end{proof}
	We are now able to prove Theorem~\ref{maintherorem:RepresentationDifferentialForms}.
	\begin{proof}[Proof of Theorem~\ref{maintherorem:RepresentationDifferentialForms}]
		It is easy to see that linear combinations of mixed Monge-Amp\`ere operators of quadratic type of degree $k$ span a non-trivial $\GL(n,\R)$-invariant subspace of $\MAVal_k(\R^n)$, so this space coincides with $\MAVal_k(\R^n)$ by Theorem~\ref{maintheorem:IrreducibilityMAVal}. Similarly, the space spanned by $\Psi_\tau$ for $\tau\in \Lambda^{n-k,k}$ is $\GL(n,\R)$-invariant by Corollary~\ref{corollary:equivarianceMeasuresFromDifferentialForms} and thus coincides with $\MAVal_k(\R^n)$. Finally, for any linear combination $P$ of $(k\times k)$-minors, there exists a differential form $\tau\in P\Lambda^{n-k,k}$ with $P=P_\tau$ by Proposition~\ref{proposition:ImageP}. By Proposition~\ref{proposition:C2FormulaCurvatureMeasurePolynomial}, $\Psi_\tau$ satisfies
		\begin{align*}
			\Psi_\tau(f)[B]=\int_{B}P(D^2(x))d\vol_n(x)\quad\text{for bounded Borel sets}~B
		\end{align*}
		 for $f\in\Conv(\R^n,\R)\cap C^2(\R^n)$. On the other hand, such a representation holds for every differential form $\tau\in P\Lambda^{n-k,k}$ by Proposition~\ref{proposition:C2FormulaCurvatureMeasurePolynomial}, so the claim follows from the previous discussion.
	\end{proof}

\section{The cone of non-negative valuations}
\label{section:positivity}	

\begin{lemma}
	\label{lemma:charPositiveHomogeneous}
	The following are equivalent for $\Psi_\tau\in\MAVal_k(\R^n)$, $\tau\in P\Lambda^{n-k,k}$:
	\begin{enumerate}
		\item $\Psi_\tau$ is non-negative.
		\item $Q[\Psi_\tau]=Q_\tau\in\P((\C^n)^k)$ is non-negative on $(\R^n)^k$.
		\item $\KL_{\Psi_\tau}\ge 0$ on $\Gr_k(\R^n)$.
		\item $P_\tau\in\P(SM_n)$ is non-negative on the subspace of real positive semi-definite matrices.
	\end{enumerate}
\end{lemma}
\begin{proof}
	1. $\Rightarrow$ 2.: For $w_1,\dots,w_k\in\R^n$, Proposition~\ref{prop:FormulaQFourier-LaplaceDiffForm} implies
	\begin{align*}
		\mathcal{F}(\GW(\Psi[\phi]))[iw_1,\dots,iw_k]=\frac{1}{k!}Q[\Psi] (w_1,\dots,w_k)\int_{\R^n}\phi(x)\exp\left(-\left\langle\sum_{j=1}^{k}w_k,x\right\rangle\right)dx.
	\end{align*}
	On the other hand,
	\begin{align*}
		\Psi[\phi]\left(\sum_{k=j}^{k}\exp(-\langle w_i,\cdot\rangle)\right)
		=&k!\overline{\Psi[\phi]}(\exp(-\langle w_1,\cdot\rangle),\dots,\exp(-\langle w_k,\cdot\rangle))\\=&k!\GW(\Psi[\phi])[\exp(-\langle w_1,\cdot\rangle)\otimes\dots\otimes\exp(-\langle w_k,\cdot\rangle)]\\
		=&k!\mathcal{F}(\GW(\Psi[\phi]))[iw_1,\dots,iw_k]
	\end{align*}
	as the polarization vanishes on other combinations of the functions $\exp(-\langle w_i,\cdot\rangle)$, $1\le i\le k$. If $\Psi$ is a non-negative valuation, we can choose a non-negative function $\phi\in C_c(\R^n)$ to see that $Q[\Psi]$ is non-negative on $(\R^n)^k$.\\
	
	2. $\Leftrightarrow$ 3.: This follows directly from \eqref{equation:PPsiOnRealSubspace}.\\
	
	2. $\Rightarrow$ 4.: This follows by considering the polarization $\bar{P}_\tau$ of $P_\tau$: If $S\in SM_n$ is a real positive semi-definite matrix, then there exists an orthonormal basis $e_1,\dots,e_n$ of $\R^n$ and $\lambda_1,\dots,\lambda_n\ge 0$ such that $S=\sum_{i=1}^{n}\lambda_i e_i\cdot e_i^T$. The multilinearity of the polarization and the fact that $P_\tau$ vanishes on matrices of rank less than $k$ implies
	\begin{align*}
		P_\tau(S)=&\sum_{|\alpha|=k}\lambda^\alpha\binom{k}{\alpha}\bar{P}_\tau(e_{\alpha_1}\cdot e_{\alpha_1}^T,\dots,e_{\alpha_k}\cdot e_{\alpha_k}^T)\\
		=&\sum_{|\alpha|=k}\frac{1}{k!}\lambda^\alpha\binom{k}{\alpha}P_{\tau}\left(\sum_{i=1}^{k}e_{\alpha_i}\cdot e_{\alpha_i}^T\right)\\
		=&\sum_{|\alpha|=k}\frac{1}{k!}\lambda^\alpha\binom{k}{\alpha}\frac{1}{k!}Q_\tau(e_{\alpha_1},\dots,e_{\alpha_k})\ge 0
	\end{align*}
	Here we have used the relation between $P_\tau$ and $Q_\tau$ from Proposition~\ref{prop:relationPQ} and that $Q_\tau=Q[\Psi]$ is non-negative.\\
	
	4. $\Rightarrow$ 1.:
	Proposition~\ref{proposition:C2FormulaCurvatureMeasurePolynomial} implies for $f\in \Conv(\R^n,\R)\cap C^2(\R^n )$
	\begin{align*}
		\Psi_\tau(f)[B]=\int_{B}P_\tau(D^2f(x))d\vol(x)\quad \text{for all bounded Borel sets}~B\subset\R^n.
	\end{align*}
	If $P_\tau$ is non-negative on the space of positive semi-definite matrices, then $\Psi_\tau(f)$ is a non-negative measure for all $f\in\Conv(\R^n,\R)\cap C^2(\R^n)$ and thus for all elements of $\Conv(\R^n,\R)$ by continuity. 
\end{proof}

\begin{corollary}
	Every $\Psi\in\MAVal_k(\R^n)$ can be written as a linear combination of non-negative valuations.
\end{corollary}
\begin{proof}
	By Corollary~\ref{corollary:PPsi_kMinors}, $Q[\Psi]\in\mathcal{M}^2_k$, that is, $Q[\Psi]$ is a linear combination of quadratic products of $(k\times k)$-minors. If $\Delta_\alpha,\Delta_\beta$ are two of these minors, then 
	\begin{align*}
		\Delta_\alpha\cdot \Delta_\beta=\frac{1}{2}(\Delta_\alpha+\Delta_\beta)^2-\frac{1}{2}(\Delta_\alpha-\Delta_\beta)^2
	\end{align*}
	is a linear combination of squares of sums of $(k\times k)$-minors. Thus $Q[\Psi]$ can be written as a linear combination of squares of sums of $(k\times k)$-minors. Consider one of the elements $\tilde{\Psi}\in \MAVal_k(\R^n)$ corresponding to one of these squares of sums of $(k\times k)$-minors, which exist by Corollary~\ref{corollary:CharacterizationInTermsOfCharacteristicPolynomial}. Then $Q[\tilde{\Psi}]$ is a square of a sum of $(k\times k)$-minors and thus non-negative on $(\R^n)^k$. In particular, $\tilde{\Psi}$ is a non-negative valuation by Lemma~\ref{lemma:charPositiveHomogeneous}. Consequently, $\Psi$ can be written as a linear combination of non-negative valuations.
\end{proof}
\begin{proposition}
	$\Psi\in\MAVal(\R^n)$ is non-negative if and only if its homogeneous components are non-negative.
\end{proposition}
\begin{proof}
	Decompose $\Psi=\sum_{k=0}^{n}\Psi_k$ into its homogeneous components. If the homogeneous components $\Psi_k$ are non-negative, then $\Psi$ is obviously non-negative. For the converse, note first that $\Psi_0$ is constant and
	\begin{align*}
		\Psi_0(0)=\Psi(0)
	\end{align*}
	is a non-negative measure. For $1\le k\le n$, it is sufficient to show that $\KL_{\Psi_k}$ is non-negative due to Lemma~\ref{lemma:charPositiveHomogeneous}. Proposition~\ref{proposition:KlainEmbedding} shows that $\KL_{\Psi_k}$ is non-negative if and only if  $\Psi_k$ is non-negative on all pullbacks of functions defined on $k$-dimensional subspaces $E\in\Gr_k(\R^n)$. For $f\in\Conv(E,\R)$,
	\begin{align*}
		\Psi(\pi_E^*f)=\sum_{i=0}^{k}\Psi_i(\pi_E^*f),
	\end{align*}
	because $\Psi_l$, $k+1\le l\le n$, vanishes on pullbacks of functions defined on subspaces of dimension less than $l$. As $\Psi(t\pi_E^*f)$ is a non-negative measure for all $t>0$ by assumption, we obtain for all bounded Borel sets $B\subset\R^n$
	\begin{align*}
		0\le \frac{1}{t^k}\Psi(t\pi_E^*f)[B]=\Psi_k(\pi_E^*f)[B]+\frac{1}{t^k}\sum_{i=0}^{k-1}t^i\Psi_i(\pi_E^*f)[B]\quad\text{for}~t>0,
	\end{align*}
	which implies $\Psi_k(\pi_E^*f)[\phi]\ge 0$ by taking the limit  $t\rightarrow\infty$. Thus $\Psi_k(\pi_E^*f)$ is a non-negative measure.
\end{proof}

	\bibliography{literature_MA.bib}

\Addresses
\end{document}